\documentclass[12pt]{article}

\usepackage{latexsym,amssymb,amsfonts}
\usepackage{amsmath}\usepackage{epsfig}

\newtheorem{theorem}{Theorem}
\newtheorem{corollary}[theorem]{Corollary}
 \newtheorem{lemma}[theorem]{Lemma}

\newtheorem{definition}[theorem]{Definition}

\newenvironment{proof}{\begin{trivlist}
    \item[\hskip\labelsep{\bf Proof.}]}{$\hfill\Box$\end{trivlist}}

\newcommand{\Vol}{{\rm Vol}}

\newcommand{\real}{\mathbb{R}}
\newcommand{\Pcal}{\mathcal{P}}

\newcommand{\rd}{\, {\rm d}}
\newcommand{\cS}{S}

\newcommand{\bsd}{\boldsymbol{d}}
\newcommand{\bsx}{\boldsymbol{x}}
\newcommand{\bsy}{\boldsymbol{y}}
\newcommand{\bsh}{\boldsymbol{h}}
\newcommand{\bsj}{\boldsymbol{j}}
\newcommand{\bsk}{\boldsymbol{k}}
\newcommand{\bsl}{\boldsymbol{l}}
\newcommand{\bsr}{\boldsymbol{r}}
\newcommand{\bsell}{\boldsymbol{l}}
\newcommand{\bsa}{\boldsymbol{a}}
\newcommand{\bsq}{\boldsymbol{q}}
\newcommand{\bsp}{\boldsymbol{p}}

\newcommand{\bszero}{\boldsymbol{0}}

\newcommand{\bsgamma}{\boldsymbol{\gamma}}
\newcommand{\bsalpha}{\boldsymbol{\alpha}}

\newcommand{\bsone}{\boldsymbol{1}}
\newcommand{\uu}{u}

\newcommand{\walb}{{\rm wal}}

\allowdisplaybreaks

\title{Quasi-Monte Carlo numerical integration on $\mathbb{R}^s$: digital nets and worst-case error}

\author{Josef Dick\footnote{School of Mathematics and Statistics, The University of New South Wales, NSW 2052, Australia; email: josef.dick@unsw.edu.au}}

\begin{document}

\maketitle

\begin{abstract}
Quasi-Monte Carlo rules are equal weight quadrature rules defined
over the domain $[0,1]^s$. Here we introduce quasi-Monte Carlo type
rules for numerical integration of functions defined on
$\mathbb{R}^s$. These rules are obtained by way of some
transformation of digital nets such that locally one obtains qMC
rules, but at the same time, globally one also has the required
distribution. We prove that these rules are optimal for numerical
integration in spaces of bounded fractional variation. The analysis
is based on certain tilings of the Walsh phase plane. Numerical
results demonstrate the efficiency of the method.
\end{abstract}



\section{Introduction}

Traditionally, quasi-Monte Carlo (qMC) rules are equal weight quadrature formulae defined on $[0,1]^s$, i.e.,
\begin{equation*}
Q_{P}(f) = \frac{1}{N} \sum_{n=0}^{N-1} f(\bsx_n),
\end{equation*}
where $P= \{\bsx_0,\ldots, \bsx_{N-1}\} \subset [0,1]^s$ are quadrature points. These rules are used to approximate integrals of the form
\begin{equation*}
\int_{[0,1]^s} f(\bsx) \, {\rm d} \bsx,
\end{equation*}
see \cite{DP09, niesiam} for more information.

In practice, however, the integrals one needs to approximate are usually over domains other than $[0,1]^s$, often $\mathbb{R}^s$. In order for qMC rules to be used in this case, one requires a transformation mapping from $\mathbb{R}^s$ to $[0,1]^s$ \cite{KSWW}. This step is of great importance, indeed different transformations can yield very different results \cite{KDSWW}. Furthermore, usually known theory does not predict which transformations from $\mathbb{R}^s$ to $[0,1]^s$ will yield the best results~\cite{KDSWW}.

To circumvent this problem, we introduce quasi-Monte Carlo type
rules for $\mathbb{R}^s$. Admittedly, there are still choices to be
made by the user of our method. In fact, one could view the approach
proposed here as a discretized transformation, where the transformed
point set yields locally qMC rules, but globally has a given
distribution.

To understand what properties a point set locally and globally must
satisfy, we analyze the worst-case error of numerical integration of
functions on $\mathbb{R}^s$ with bounded fractional variation of order $\alpha$ and
which satisfy a decay condition as one moves away from the origin.
Roughly speaking, if $\alpha = 1$ then the functions we consider have square integrable partial mixed derivatives up to order one in each coordinate. That is, for a partition $\mathcal{P}_{\mathbb{R}^s}$ of $\mathbb{R}^s$ into subintervals $J=\prod_{i=1}^s [a_i, b_i)$ we have
\begin{equation*}
\sum_{J \in \mathcal{P}_{\mathbb{R}^s}} V^2_J(f) <\infty,
\end{equation*}
where $$V_J(f) = \left(\int_J \left|\frac{\partial^s f}{\partial x_1
\cdots \partial x_s} \right|^2\,\mathrm{d} \boldsymbol{x}
\right)^{1/2}.$$ The value $V_{J}(f)$ depends on the function as
well as the location $J$. We assume that for $J=\prod_{i=1}^s
[a_i,b_i)$ we have $V_{J}(f) \rightarrow 0$ as $\max_{1\le i \le s}
\min(|a_i|,|b_i|) \rightarrow \infty$ with a certain rate of decay.
Furthermore, also the function $f$ itself must satisfy a certain
rate of decay as $\|\bsx\| \rightarrow \infty$ (where $\|\cdot \|$
denotes some norm in $\mathbb{R}^s$).

For these functions, we now construct quadrature rules for which one obtains an integration error bounded above by $N^{-\alpha} (\log N)^{c(s,\alpha)}$ (for some $c_{s,\alpha} > 0$ depending only on $s$ and $\alpha$), where $N$ denotes the number of quadrature points and $s$ the dimension. The focus in this paper is on the asymptotic convergence rate as $N \rightarrow \infty$. The study of tractability \cite{NovWoz} is left for future work.

The design of optimal quadrature rules depends on several properties
of the functions considered. One is the smoothness of the functions.
If the function is concentrated on a bounded region, then the
quadrature rule should be able to integrate functions with
smoothness $\alpha$ locally with sufficient accuracy. In other
words, one needs enough flexibility to have optimal quadrature rules
locally. Further, since we can only use a finite number of
quadrature points $N$, for fixed $N$ we can only cover a finite
region. Here the assumption of the decay of $f$ and the condition on
$V_J(f)$ comes in. It tells us where to concentrate our efforts. The
local property allows us to design quadrature rules which are
optimal over subcubes of the form $\prod_{i=1}^s [a_i, b_i]$, the
global property tells us how important each of those subcubes is and
how their importance is distributed in the space $\mathbb{R}^s$. For
example, assume that $f$ and $V_{J}(f)$ decay at least like
$\prod_{i=1}^s (1+ |x_i|)^{-1}$ as $\|\bsx\| \rightarrow \infty$,
where $\bsx = (x_1,\ldots, x_s)$. Then the most important subcubes
are arranged in a hyperbolic cross around the origin.

Below we show how one can map a digital net defined on $[0,1]^s$ such that in each of the subcubes, inside this hyperbolic cross region, one obtains a digitally shifted digital net, and at the same time, the density of the points in each subcube decreases according to the rate of decay of $f$ and $V_{J}(f)$. Our approach is flexible enough to also be able to handle other rates of decay of $f$ and $V_{J}(f)$. Another way to think about this procedure is the following: discretize the transformation from $[0,1]^s$ to $\mathbb{R}^s$ such that a digital net in $[0,1]^s$ is mapped to several nets in subcubes of $\mathbb{R}^s$, such that the subcubes and number of points therein are distributed according to a given decay rate.

The cost of constructing the point set in $\mathbb{R}^s$ is similar
to the cost of constructing the underlying digital net over a finite
field $\mathbb{Z}_b$, assuming that the one-dimensional projections
have already been defined. To be more precise, instead of using a
mapping $\mathbb{Z}_b^m \mapsto [0,1)$ as is done for digital nets,
one uses a look-up table which defines a mapping $\mathbb{Z}_b^m
\mapsto \mathbb{R}$. The cost of constructing the vectors in
$\mathbb{Z}_b^m$ is the same for digital nets in $[0,1)^s$ and in
$\mathbb{R}^s$. The one-dimensional projections on the other hand
need to be carefully designed in advance, as we illustrate in
Section~\ref{sec_examples_all}. Since those are designed in advance,
we do not count it towards the construction cost of an individual
point set.

A numerical test in Matlab, see Section~\ref{sec_numresult}, reveals
that the computation using the method introduced here is considerably faster
than using digital nets and the inverse normal cumulative
distribution function (in the numerical example considered here, the proposed method is more than ten times faster). The advantage is that in the proposed method
one does not need to compute the inverse normal cumulative
distribution function, which can be time consuming.

Conceivably one could also create the global property by hand: Take a number of digital nets $P_1,\ldots, P_K$ (for instance from a digital sequence),  partition $\mathbb{R}^s$ into subcubes $C_1,\ldots, C_K$, and put the digital net $P_k$ in $C_k$ for $1 \le k \le K$. In this case, the global property is simply designed by hand (or even adaptively). This method can be become very involved if the dimension $s$ is large. Our approach on the other hand builds on a tensor product structure, which simplifies matters, especially in high dimensions. To fit, or nearly fit, the integrand to such a structure it may be advantageous
to use the method proposed here in conjunction with some variance
reduction method like principal component analysis, Brownian bridge
constructions or similar methods \cite{Lem}.

We give an overview of the paper. The analysis of the integration
error is based on time-frequency analysis. To that end, we introduce
tilings of the Walsh phase plane \cite{Thiele,Thiele2}, which gives
us the flexibility to consider subcubes of different size and adjust
them to the rate of decay of the function $f$ and the variation
$V_{J}(f)$. This is done in Section~\ref{sec_Walshmodel}. In
Section~\ref{sec_decay_Walsh} we analyze the rate of decay of the
Walsh coefficients based on the smoothness of $f$, the rate of decay
of $f$, and, roughly speaking, the rate of decay of $V_{J}(f)$.

In Section~\ref{sec_num_int} we analyze the integration error: We prove an upper bound on the integration error which consists of two parts. One part deals with the remainder, that is, the region where there are no quadrature points. For this part one relies on the the rate of decay of $f$ to obtain an upper bound. The other part is concerned with the integration error one commits when using a quadrature rule. For this part we use estimates of the integration error for qMC rules based on digital nets.

The construction of the quadrature points is introduced in
Section~\ref{sec_construction}. We show how one can map a digital
net in $[0,1]^s$ into subcubes of $\mathbb{R}^s$ such that each
subcube contains a digitally shifted digital net, where the size of
each digitally shifted digital net is according to a given
distribution. In Section~\ref{sec_error_bound} we give upper bounds
on the integration error when one uses the construction of
Section~\ref{sec_construction}. We show that under certain decay
rates of $f$ and $V_{J}(f)$, the convergence of the integration
error is optimal up to some power of $\log N$. In
Section~\ref{sec_examples} we provide three concrete examples of how
our bounds can be used in a few situations. In
Section~\ref{sec_numresult} we represent some numerical results
which demonstrate the efficiency of the method.


For the convenience of the reader we introduce some notation used throughout the paper in the following subsection.

\subsubsection*{Notation}

For $s \ge 1$ let $S = \{1,\ldots, s\}$. For $u \subseteq S$ we denote by $|u|$ the number of elements in $u$.

The set of complex number is denoted by $\mathbb{C}$, the set of
real numbers is denoted by $\mathbb{R}$, the set of integers by
$\mathbb{Z}$, the set of natural numbers by $\mathbb{N}$, the set of
nonnegative integers by $\mathbb{N}_0$, and the finite field of
prime order $b$ by $\mathbb{Z}_b$. Further let $\mathbb{R}_0^+ = \{x
\in \mathbb{R}: x \ge 0\}$ and $\mathbb{R}^+ = \{x \in \mathbb{R}: x
> 0\}$. For $c \in \mathbb{C}$ we write $\overline{c}$ for the
complex conjugate of $c$. For $k, l \in \mathbb{Z}$ we write $k | l$
if $k$ divides $l$.

We always write vectors $\bsh = (h_1,\ldots, h_s), \bsj = (j_1,\ldots, j_s), \bsk = (k_1,\ldots, k_s)$, and so on. Further we set $\bszero = (0,\ldots, 0)$ and $\boldsymbol{1} = (1,\ldots, 1)$ - the dimension of these vectors is apparent from the context in which they occur. Further, for $u \subseteq S$ we set $\bsh_u = (h_i)_{i\in u}$ and $(\bsh_u, \bszero)$ is the vector whose $i$th coordinate is $h_i$ for $i \in u$ and $0$ otherwise. For short we write $b^{\bsj} = (b^{j_1}, \ldots, b^{j_s})$ and $b^{-\bsj} = (b^{-j_1}, \ldots, b^{-j_s})$. For vectors $\bsk, \bsl$ in $\mathbb{R}^s$, $\mathbb{Z}^s$, $\mathbb{N}^s$, and so on, we set $\bsk \star \bsl = (k_1 l_1, \ldots, k_s l_s)$. In particular we write $b^{\bsj} \star \bsl = (b^{j_1} l_1, \ldots, b^{j_s} l_s)$. Further, for $\bsl = (l_1,\ldots, l_s)$ we set
\begin{equation*}
|\bsl|_1 = |l_1| + \cdots + |l_s|.
\end{equation*}
Further we write $|\bsx|_\infty = \max_{1 \le i \le s} |x_i|$ and $|\bsx|
= (|x_1|, \ldots, |x_s|)$.

For $\bsr \in \mathbb{N}_0^s$ we set
\begin{equation*}
u_{\bsr} = \{i \in S: r_i \neq 0\}.
\end{equation*}

We define the $L_2$ inner product for $f, g \in L_2(\mathbb{R}^s)$ as usual by
\begin{equation*}
\langle f, g \rangle_{L_2} = \int_{\mathbb{R}^s} f(\bsx) \overline{g(\bsx)} \, {\rm d} \bsx.
\end{equation*}
We say that $f$ is orthogonal to $g$ if $\langle f, g \rangle_{L_2} = 0$.

For $J \subset \mathbb{R}^s$ we define the characteristic function by
\begin{equation*}
1_{J}(\bsx) = \left\{\begin{array}{ll} 1 & \mbox{if } \bsx \in J, \\ 0 & \mbox{otherwise}. \end{array} \right.
\end{equation*}

Let $P = \{\bsx_0,\ldots, \bsx_{N-1}\} \subset \mathbb{R}^s$, $\Lambda = \{\lambda_0, \ldots, \lambda_{N-1}\} \subset \mathbb{R}$. We define the quadrature rule
\begin{equation*}
Q_{P,\Lambda}(f) = \sum_{n=0}^{N-1} \lambda_n f(\bsx_n).
\end{equation*}

For $k \in \mathbb{N}$ with base $b$ representation $k = \kappa_0 +
\kappa_1 b + \cdots + \kappa_{a-1} b^{a-1}$, we define the vector $\vec{k} =
(\kappa_0,\ldots, \kappa_{m-1})^\top \in \mathbb{Z}_b^m$, where we
set $\kappa_{a+1} = \cdots = \kappa_{m-1} = 0$ if $a < m$.

\section{The Walsh model}\label{sec_Walshmodel}

For the convenience of the reader we repeat some elementary results
concerning the Walsh model, see \cite{Thiele, Thiele2} for more
detailed information, which is based on Walsh functions, see
\cite{Fine, Walsh}. Let $b \ge 2$ be an integer, $k = \kappa_0 + b
\kappa_1 + \cdots + \kappa_{a} b^{a} \in \mathbb{N}_0$, $x = x_c
b^{c} + x_{c-1} b^{c-1} + \cdots \in \mathbb{R}$ for some $a, c \in
\mathbb{N}_0$, be the base $b$ representations of $k$ and $x$. Then
we define the $k$th Walsh function by
\begin{equation*}
{\rm wal}_k(x) = \omega_b^{\kappa_0 x_{-1} + \cdots + \kappa_a x_{-a-1}} \, 1_{[0,1)}(x),
\end{equation*}
where $\omega_b = {\rm e}^{2 \pi {\rm i}/b}$.

We define translations and dilations of ${\rm wal}_k$. For $j, l \in \mathbb{Z}$ let
\begin{equation*}
w_{j,k,l} = b^{-j/2} {\rm wal}_{k}(b^{-j} x - l).
\end{equation*}
Note that the support of $w_{j,k,l}$ is given by $[b^j l, b^j (l+1))$.

The system $\{w_{j,k,l}: k \in \mathbb{N}_0, j, l \in \mathbb{Z}\}$
is overdetermined in $L_2(\mathbb{R})$. Nonetheless, one can
identify subsets which are complete orthonormal systems in
$L_2(\mathbb{R})$.

Let $j,k,l$ be integers with $k \ge 0$. To each function $w_{j,k,l}$ there corresponds a tile $T$ given by
\begin{equation*}
T = T_{j,k,l} = [ b^j l, b^j (l+1)) \times [b^{-j} k, b^{-j} (k+1)).
\end{equation*}

For the convenience of the reader we prove some results concerning
the orthogonality and completeness of Walsh functions. See also
\cite{Thiele} and the references therein, where these results and
further information can be found. The following results will be
sufficient for our purposes here.

\begin{lemma}\label{lem_orthogonal}
Let $j, j', k, k', l, l'$ be integers such that $k, k' \ge 0$. Then $w_{j,k,l}$ is orthogonal to $w_{j', k', l'}$ if and only if $T_{j,k,l} \cap T_{j', k', l'} = \emptyset$.
\end{lemma}

\begin{proof}
Assume that $$T_{j,k,l} \cap T_{j',k',l'} = \emptyset.$$ If $$[b^j l, b^j (l+1)) \cap [b^{j'} l', b^{j'} (l'+1)) = \emptyset,$$ then the functions $w_{j,k,l}$ and $w_{j',k',l'}$ have disjoint support and hence are orthonormal. Assume now that $$[b^j l, b^j (l+1)) \cap [b^{j'} l', b^{j'} (l'+1)) =  [\max\{b^{j}l, b^{j'}l'\}, \min\{b^j(l+1), b^{j'} (l'+1)\}) \neq \emptyset$$ and
\begin{equation}\label{kint}
[b^{-j} k, b^{-j} (k+1)) \cap [b^{-j'} k', b^{-j'} (k'+1)) = \emptyset.
\end{equation}

We consider the case $j \le j'$ (the other case can be shown
analogously). Then $[b^{-j} k, b^{-j}(k+1)) \subseteq [b^{-j'}k', b^{-j'}(k'+1))$ and therefore $$\max\{b^jl, b^{j'} l'\} = b^j l$$ and
$$\min\{b^j(l+1), b^{j'}(l'+1)\} = b^j (l+1).$$ Therefore, using the
substitution $x = b^j (y+l)$, we obtain
\begin{eqnarray}\label{innerproductww}
\int_{\mathbb{R}} w_{j,k,l}(x) \overline{w_{j',k', l'}(x)} \, {\rm
d} x & = & b^{-j/2 - j'/2} \int_{b^j l}^{b^{j} (l+1)} {\rm
wal}_k(b^{-j} x - l)
\overline{{\rm wal}_{k'}(b^{-j'} x - l')} \, {\rm d} x \nonumber \\
& = & b^{(j-j')/2} \int_0^1 {\rm wal}_k(y) \, \overline{{\rm
wal}_{k'}(y b^{j-j'} + b^{j-j'} l - l')} \, {\rm d} y \nonumber \\ &
= & b^{(j-j')/2} \overline{{\rm wal}_{k'}(b^{j-j'} l - l')} \int_0^1
{\rm wal}_{k}(y) \, \overline{{\rm wal}_{k'}(y b^{j-j'})} \, {\rm d}
y \nonumber \\ & = & b^{(j-j')/2} \overline{{\rm wal}_{k'}(b^{j-j'}
l - l')} \int_0^1 {\rm wal}_{k}(y) \, \overline{{\rm wal}_{\lfloor
k' b^{j-j'} \rfloor}(y)} \, {\rm d} y.
\end{eqnarray}

From \eqref{kint} it follows for $j \le j'$ we either have $b^{-j'}
k' \ge b^{-j} (k+1)$ or $b^{-j'}(k'+1) \le b^{-j} k$. In both cases
we have $k \neq \lfloor k' b^{j-j'} \rfloor$ and hence
\eqref{innerproductww} yields $0$.

Now assume that $T_{j,k,l} \cap T_{j',k',l'} \neq \emptyset$. Then the support of $w_{j,k,l}$ and $w_{j',k',l'}$ is not disjoint. Using the same arguments as above we arrive at \eqref{innerproductww}. Assuming again $j \le j'$, we have $[b^{-j'} k', b^{-j'}(k'+1)) \subseteq [b^{-j} k, b^{-j} (k+1))$ which implies $k = \lfloor k' b^{j-j'} \rfloor$. Hence \eqref{innerproductww} yields
\begin{equation}\label{innerproductww2}
\int_{\mathbb{R}} w_{j,k,l}(x) \, \overline{w_{j',k',l'}(x)} \, {\rm
d} x = b^{(j-j')/2} \overline{{\rm wal}_{k'}(b^{j-j'} l - l')} \neq
0
\end{equation}
and the result follows.
\end{proof}

\begin{lemma}\label{lem_bro}
Let $j,k,l \in \mathbb{Z}$ such that $k \ge 0$ and $b | l$. Then
\begin{equation*}
{\rm span} \, \{w_{j,k,l}, w_{j,k,l+1}, \ldots, w_{j,k, l+b-1}\} =
{\rm span} \, \{w_{j+1,k b,l/b}, w_{j+1, k b + 1, l/b}, \ldots,
w_{j+1, kb + b-1, l/b}\}.
\end{equation*}
\end{lemma}

\begin{proof}
Using \eqref{innerproductww2} we obtain for $0 \le r < b$ that
\begin{equation*}
\sum_{s=0}^{b-1} \langle w_{j+1,kb+r,l/b}, w_{j,k,l+s}\rangle_{L_2} \,
w_{j,k,l+s}(x) = b^{-1/2} \sum_{s=0}^{b-1} {\rm wal}_r(s/b)
w_{j,k,l+s}(x) = w_{j+1,kb+r,l/b}(x)
\end{equation*}
and
\begin{equation*}
\sum_{s=0}^{b-1} \langle w_{j,k,l+r}, w_{j+1,kb+s,l/b} \rangle_{L_2} \,
w_{j+1,k b + s, l/b}(x)  = b^{-1/2} \sum_{s=0}^{b-1} \overline{{\rm
wal}_s(r/b)} w_{j+1,kb+s,l/b}(x)  =  w_{j,k,l+r}(x),
\end{equation*}
hence the result follows.
\end{proof}

The dual of Lemma~\ref{lem_bro} in terms of the corresponding tiles
can be stated in the following manner. Let $j,k,l$ be as in
Lemma~\ref{lem_bro}. Then:
\begin{quote}
the tiles $T_{j,k,l}, T_{j,k,l+1},\ldots,
T_{j,k,l+b-1}$ cover the same area as

the tiles $T_{j+1,kb, l/b}, T_{j+1, kb+1, l/b}, \ldots, T_{j+1, kb+ b-1, l/b}$.
\end{quote}

\begin{lemma}\label{lem_bro2}
Let $\tau$ and $\tau'$ be two finite sets of tiles such that all pairs of tiles in $\tau$ are disjoint and all pairs of tiles in $\tau'$ are disjoint. Let $W$ and $W'$ be the corresponding sets of Walsh functions. Then
$$\bigcup_{T \in \tau} T = \bigcup_{T' \in \tau'} T' \quad \Longleftrightarrow \quad {\rm
span}\, W  = {\rm span} \, W'.$$
\end{lemma}

\begin{proof}
The proof follows by successively using Lemma~\ref{lem_bro}.
\end{proof}

\begin{lemma}
Let $J \subseteq \mathbb{Z}$, $K \subseteq \mathbb{N}_0$, and $L
\subseteq \mathbb{Z}$ and set $$\tau = \{T_{j,k,l}: j \in J, k \in
K, l \in L\}.$$ Assume that the tiles in $\tau$ are pairwise disjoint. Then the system
$$W = \{w_{j,k,l}: j \in J, k \in K, l \in L\}$$ is a complete
system in $L_2(\mathbb{R})$ if and only if $$\bigcup_{T \in \tau} T
= \mathbb{R} \times \mathbb{R}_0^+.$$
\end{lemma}

\begin{proof}
If $$\bigcup_{T \in \tau} T = \mathbb{R} \times \mathbb{R}_0^+,$$
then by applying Lemma~\ref{lem_bro2} one can obtain all tiles of
the form $T_{0,k,l}$, $k \in \mathbb{N}_0$ and $l \in \mathbb{Z}$ as a finite linear combination of tiles in $\tau$. These tiles also cover
$\mathbb{R} \times \mathbb{R}_0^+$, and the corresponding set of
Walsh functions includes the classical Walsh function system, which
is known to be complete in $L_2(\mathbb{R})$. Since the span stays
unchanged when applying Lemma~\ref{lem_bro2} it follows that the
system $W$ is a complete system in $L_2(\mathbb{R})$.

Assume that $$\bigcup_{T \in \tau} T \neq \mathbb{R} \times
\mathbb{R}_0^+.$$ Then there is a tile $T'$ such that $T' \subseteq
\mathbb{R} \times \mathbb{R}_0^+ \setminus \bigcup_{T \in \tau} T$,
then the corresponding Walsh function is orthogonal to all functions
in $W$ and hence $W$ is not complete.
%
%
\end{proof}

In the following we consider functions $f: \mathbb{R}^s \to \mathbb{R}$. In this case we define tensor products of the Walsh functions in the following way. Let $\bsj, \bsl \in \mathbb{Z}^s$ and $\bsk \in \mathbb{N}_0^s$ be given by $\bsj = (j_1,\ldots, j_s)$, $\bsl = (l_1,\ldots, l_s)$, and $\bsk = (k_1,\ldots, k_s)$. Then
\begin{equation*}
w_{\bsj,\bsk,\bsl}(\bsx) = \prod_{i=1}^s w_{j_i, k_i, l_i}(x_i).
\end{equation*}
The tile corresponding to $w_{\bsj,\bsk,\bsl}$ is given by
\begin{equation*}
T_{\bsj,\bsk,\bsl} = \prod_{i=1}^s T_{j_i, k_i, l_i} = \prod_{i=1}^s \left([b^{j_i} l_i, b^{j_i}(l_i+1)) \times [b^{-j_i} k_i, b^{-j_i} (k_i+1))\right).
\end{equation*}
All the results of this section also hold for the tensor product case, that is, two functions $w_{\bsj, \bsk,\bsl}, w_{\bsj', \bsk', \bsl'}$ are orthogonal if and only if the corresponding tiles are disjoint and a system of orthogonal functions $\{w_{\bsj,\bsk,\bsl}: (\bsj,\bsk,\bsl) \in R\}$ is complete in $L_2(\mathbb{R}^s)$ if and only if the corresponding tiles cover $(\mathbb{R} \times\mathbb{R}_0^+)^s$.

\section{Smoothness, convergence behavior, and the decay of the Walsh coefficients}\label{sec_decay_Walsh}

In this section we define classes of integrands of functions $f:\mathbb{R}^s \to \mathbb{R}$. The smoothness of the integrands will be controlled by local smoothness parameters and the rate of decay of $f(\bsx)$, as the point $\bsx$ tends to infinity (in one or more of its coordinates), is controlled by local weight parameters.

Let $s \ge 1$ and let $\Pcal_{\mathbb{R}^s}$ be a partition of $\mathbb{R}^s$ into subintervals $J$ of the form $J = [b^{\bsj} \star \bsell, b^{\bsj} \star (\bsell + \boldsymbol{1})) \subseteq \mathbb{R}^s$, where $\bsj,\bsell \in \mathbb{Z}^s$.

We control the rate of decay of the Walsh coefficients of the
integrand $f$ using three parameters:
\begin{itemize}
\item[(i)] The local smoothness of the integrand $f$, denoted by $\alpha_u$ for $\emptyset \neq u \subseteq S$ (we assume $1/2 < \alpha_u \le 1$);
\item[(ii)] The rate of decay of the integrand $f$ as
$|\bsx|_\infty \rightarrow \infty$; for this we use a function
$\gamma_{\emptyset}: \Pcal_{\mathbb{R}^{s}} \to \mathbb{R}^+$;
\item[(iii)] The rate of decay of the `derivative' of the integrand
$f$; for this we use functions $\gamma_u: \Pcal_{\mathbb{R}^{s}} \to
\mathbb{R}^+$ for $\emptyset \neq u \subseteq S$;
\end{itemize}

We introduce some necessary restrictions on the functions
$\gamma_u$.

\begin{definition}\label{def_weightfun}
Let $s \ge 1$ and let $\Pcal_{\mathbb{R}^s}$ be a partition of $\mathbb{R}^s$ into subintervals $J$ of the form $J = [b^{\bsj} \star \bsell, b^{\bsj} \star (\bsell + \boldsymbol{1})) \subseteq \mathbb{R}^s$, where $\bsj,\bsell \in \mathbb{Z}^s$.

Then we call $\bsgamma = (\gamma_u)_{u\subseteq S}$
local weight parameters if the weight functions $\gamma_u:
\Pcal_{\mathbb{R}^s} \to \mathbb{R}^+$ are such that for each $u \subseteq S$ we have
$$\sup_{J \in \Pcal_{\mathbb{R}^s}} \gamma_u(J) <
\infty.$$

We call $\bsalpha = (\alpha_u)_{\emptyset \neq u \subseteq S}$ local
smoothness parameters if the functions $\alpha_u: \Pcal_{\mathbb{R}^s} \to
\mathbb{R}^+$ are such that
$$\sup_{J \in \Pcal_{\mathbb{R}^s}} \alpha_u(J) \le 1 \quad \mbox{and} \quad \inf_{J \in \Pcal_{\mathbb{R}^{s}}} \alpha_u(J) > 1/2$$ for $\emptyset \neq u \subseteq S$.
\end{definition}
Note that the assumption $\inf_{J \in \Pcal_{\mathbb{R}^s}} \alpha_u(J) > 1/2$ is needed for the main results of the paper, hence we include it already in Definition~\ref{def_weightfun}.

We now define a local variation for functions $f:\mathbb{R}^s \to \mathbb{R}$. Let $J = [b^{\bsj} \star \bsell, b^{\bsj} \star (\bsell + \boldsymbol{1})) \subseteq \mathbb{R}^s$ for some $\bsj, \bsell \in \mathbb{Z}^s$. For a subinterval $I = \prod_{i=1}^s [x_i, y_i) \subseteq J$ with $x_i < y_i$ and a function $f:\mathbb{R}^s \rightarrow \real$, let the
function $\Delta(f,I)$ denote the alternating sum of $f$ at the
vertices of $I$ where adjacent vertices have opposite signs. (Hence, for instance,
for $f = \prod_{i=1}^s f_i$ we have $\Delta(f,I) = \prod_{i=1}^s
(f_i(x_i) - f_i(y_i))$.)

Let $J = [b^{\bsj} \star \bsell, b^{\bsj} \star (\bsell + \boldsymbol{1})) \subseteq \mathbb{R}^s$ for some $\bsj, \bsell \in \mathbb{Z}^s$. We define the local generalized variation in the sense of
Vitali of order $1/2 < \alpha \le 1$ in $J$
by
\begin{equation*}
V^{(s)}_{\alpha, J}(f) = \sup_{{\Pcal_J}} \left(\sum_{I \in \Pcal_J}
\Vol(I) \left|\frac{\Delta(f,I)}{\Vol(I)^{\alpha}}\right|^{2}\right)^{1/2},
\end{equation*}
where the supremum is extended over all partitions $\Pcal_J$ of
$J$ into subintervals and $\Vol(I)$ denotes the volume of the
subinterval $I$. (Again, one could include the cases where $0<\alpha \le 1/2$.)

For $\alpha = 1$ and if the partial derivatives of $f$ are
continuous on $J$ we also have the formula
\begin{equation*}
V_{1,J}^{(s)}(f) = \left(\int_{J} \left|\frac{\partial^s
f}{\partial x_1\cdots \partial x_s} \right|^{2} \rd
\bsx\right)^{1/2}.
\end{equation*}

For $\emptyset \neq u \subseteq\cS$, let $V_{\alpha, J}^{(|u|)}(f_u;u)$
be the local generalized Vitali variation of order $1/2 < \alpha \le
1$ in $J_u = [b^{\bsj_u} \star \bsell_u, b^{\bsj_u} \star (\bsell_u + \boldsymbol{1}_u))$ of the $|u|$-dimensional function
$$f_u(\bsx_u) = \int_{J_{S \setminus u}} f(\bsx) \rd \bsx_{\cS\setminus
u},$$  where $J_{S \setminus u} = [b^{\bsj_{S \setminus u}} \star \bsell_{S \setminus u}, b^{\bsj_{S \setminus u}} \star (\bsell_{S \setminus u} + \boldsymbol{1}_{S \setminus u}))$. For $u = \emptyset$ we define
$$V_{\alpha, J}^{(|\emptyset|)}(f_\emptyset;\emptyset) = \left(\int_{J} |f(\bsx)|^2 \rd \bsx \right)^{1/2}.$$


Let $\Pcal_{\mathbb{R}^s}$ be a partition of $\mathbb{R}^s$ into subintervals of the form $J = [b^{\bsj} \star \bsell, b^{\bsj} \star (\bsell + \boldsymbol{1}))$. The generalized Hardy and Krause variation with local smoothness $\bsalpha$ and local weight $\bsgamma$ with respect to the partition $\Pcal_{\mathbb{R}^s}$ is defined by
\begin{equation*}
V_{\bsalpha,\bsgamma}(f) = \left(\sum_{J \in \Pcal_{\mathbb{R}^s}} \sum_{u \subseteq S} \left[ \gamma^{-1}_u(J) V^{|u|}_{\alpha_u(J), J}(f_u; u) \right]^2  \right)^{1/2}.
\end{equation*}
A function $f$ for which $V_{\bsalpha,\bsgamma}(f) < \infty$ is said to be of
bounded (or finite) variation of order $\bsalpha$. Further we set
\begin{equation*}
H_{\bsalpha,\bsgamma} = \{f:\mathbb{R}^s \to \mathbb{R}: f \mbox{ is continuous and } V_{\bsalpha,\bsgamma}(f) <\infty \}.
\end{equation*}

Let $f: \mathbb{R}^s \to \mathbb{R}$ be given such that
$V_{\bsalpha, \bsgamma}(f) < \infty$. Let the Walsh coefficient
$\widehat{f}_{\bsj,\bsk,\bsl}$ be given by
\begin{equation*}
\widehat{f}_{\bsj,\bsk,\bsl} = \langle f, w_{\bsj,\bsk,\bsl} \rangle_{L_2}.
\end{equation*}
For $\bsj,\bsl \in \mathbb{Z}^s$ and $\bsr \in \mathbb{N}_0^s$ let
\begin{equation*}
\sigma_{\bsj,\bsr,\bsl}(f) = \left(\sum_{k_1=\lfloor b^{r_1-1} \rfloor}^{b^{r_1}-1} \cdots \sum_{k_s = \lfloor b^{r_s-1} \rfloor}^{b^{r_s}-1} |\widehat{f}_{\bsj,\bsk,\bsl}|^2 \right)^{1/2},
\end{equation*}
where $\bsr = (r_1,\ldots, r_s)$.

\begin{lemma}\label{lem_boundw}
Let $\Pcal_{\mathbb{R}^s}$ be a partition of $\mathbb{R}^s$ into subintervals of the form $J = [b^{\bsj} \star \bsell, b^{\bsj} \star (\bsj + \boldsymbol{1}))$, let $\bsalpha$ be local smoothness parameters, and $\bsgamma$ be local weight parameters. Let $f: \mathbb{R}^s \to \mathbb{R}$ be
given such that $V_{\bsalpha,\bsgamma}(f) < \infty$.
\begin{itemize}
\item[(i)] For any $\bsj,\bsl \in \mathbb{Z}^s$ such that $J = [b^{\bsj} \star \bsell, b^{\bsj} \star (\bsell + \boldsymbol{1})) \in \Pcal_{\mathbb{R}^s}$ and $\bsr \in \mathbb{N}_0^s \setminus \{\bszero\}$, where $u_{\bsr} = \{i \in S: r_i \neq 0\}$, we have
\begin{equation*}
\sigma_{\bsj,\bsr,\bsl}(f) \le (b-1)^{(\alpha_{u_{\bsr}}(J) - 1/2) |u_{\bsr}|} b^{-\alpha_{u_{\bsr}}(J) |\bsr|_1} b^{\alpha_{u_{\bsr}}(J) \sum_{i \in u_{\bsr}} j_i} b^{-\sum_{i \notin u_{\bsr}} j_i/2}  \gamma_u(J) V_{\bsalpha,\bsgamma}(f).
\end{equation*}

\item[(ii)] For any $\bsj, \bsl \in \mathbb{Z}$ such that $J = [b^{\bsj} \star \bsell, b^{\bsj} \star (\bsell + \boldsymbol{1})) \in \Pcal_{\mathbb{R}^s}$ and $\bsr \in \mathbb{N}_0^s$, where $u_{\bsr} = \{i \in S: r_i \neq 0\}$,  we have
\begin{equation*}
\sigma_{\bsj,\bsr,\bsl}(f) \le b^{|u_{\bsr}|} 2^{|u_{\bsr}|} \gamma_{\emptyset}(J) V_{\bsalpha,\bsgamma}(f).
\end{equation*}
\end{itemize}
\end{lemma}

\begin{proof}
We prove $(i)$. Let $\bsr \in \mathbb{N}_0^s \setminus\{\bszero\}$.
We have
\begin{eqnarray*}
\widehat{f}_{\bsj,\bsk,\bsl} & = & \int_{\mathbb{R}^s} f(\bsx) \overline{w_{\bsj,\bsk,\bsl}(\bsx)} \, {\rm d} \bsx \\ & = & b^{-(j_1+\cdots + j_s)/2} \int_{b^{j_1} l_1}^{b^{j_1} (l_1+1)} \cdots \int_{b^{j_s} l_s}^{b^{j_s}(l_s+1)} f(\bsx) \overline{{\rm wal}_{\bsk}(b^{-\bsj} \star \bsx - \bsl)} \, {\rm d} \bsx \\ & = & b^{(j_1 + \cdots + j_s) /2} \int_{[0,1]^s} f(b^{\bsj} \star (\bsy+\bsl)) \overline{{\rm wal}_{\bsk}(\bsy)} \, {\rm d} \bsy,
\end{eqnarray*}
where $b^{-\bsj} \star \bsx = (b^{-j_1} x_1, \ldots, b^{-j_s} x_s)$. We have $u_{\bsr} = \{i \in S: r_i \neq 0\} \neq \emptyset$ and therefore, using \cite[Lemma~13.23]{DP09}, we have
\begin{eqnarray*}
\sigma^2_{\bsj,\bsr,\bsl}(f) & \le & (b-1)^{(2\alpha - 1) |u_{\bsr}| } b^{j_1 + \cdots + j_s} b^{-2\alpha |\bsr|_1} \left[ V^{(|u_{\bsr}|)}_{\alpha,[0,1]^s}(f_{u_{\bsr}}(b^{\bsj} \star (\cdot + \bsell)); u_{\bsr}) \right]^2 \\ & = & (b-1)^{(2\alpha - 1) |u_{\bsr}| } b^{j_1 + \cdots + j_s} b^{-2\alpha |\bsr|_1} b^{- (1-2 \alpha)\sum_{i \in u_{\bsr}} j_i } b^{-2\sum_{i \notin u_{\bsr}} j_i} \left[ V^{(|u_{\bsr}|)}_{\alpha,J}(f_{u_{\bsr}}; u_{\bsr}) \right]^2 \\ & = & (b-1)^{(2\alpha - 1) |u_{\bsr}|} b^{-2 \alpha |\bsr|_1} b^{2\alpha \sum_{i \in u_{\bsr}} j_i}  b^{-\sum_{i \notin u_{\bsr}} j_i} \left[ V^{(|u_{\bsr}|)}_{\alpha,J}(f_{u_{\bsr}}; u_{\bsr}) \right]^2 \\ & \le & (b-1)^{(2\alpha - 1) |u_{\bsr}|} b^{-2 \alpha |\bsr|_1} b^{2\alpha \sum_{i \in u_{\bsr}} j_i} b^{-\sum_{i \in u_{\bsr}} j_i}  \gamma_u(J)^{2} V^2_{\bsalpha,\bsgamma}(f).
\end{eqnarray*}

We now show $(ii)$ for $\bsr = \bszero$. Notice that for $\bsk = 0$ we have $$w_{\bsj,\bszero,\bsl}(\bsx) = b^{-(j_1 + \cdots + j_s)/2} 1_{[b^{\bsj} \star \bsl, b^{\bsj} \star (\bsl+\boldsymbol{1} ))}(\bsx),$$ hence
\begin{eqnarray*}
\sigma_{\bsj,\bszero,\bsell} & = & |\widehat{f}_{\bsj,\bszero,\bsl}|  =   \left|\int_{\mathbb{R}^s}
f(\bsx) \overline{w_{\bsj,\bszero,\bsl}(\bsx)} \, {\rm d} \bsx
\right| \\ & = &  b^{-(j_1+\cdots + j_s) /2} \left|\int_{[b^{\bsj}
\star \bsl, b^{\bsj} \star (\bsl + \boldsymbol{1}))} f(\bsx) \, {\rm
d} \bsx \right| \\ & \le &  b^{-(j_1 + \cdots + j_s)/2} \int_{[b^{\bsj} \star \bsell, b^{\bsj} \star (\bsell + \boldsymbol{1}))} |f(\bsx)| \rd \bsx \\ & \le & b^{-(j_1 + \cdots + j_s)/2} \left( \int_{[b^{\bsj} \star \bsell, b^{\bsj} \star (\bsell + \boldsymbol{1}))} 1 \rd \bsx \right)^{1/2} \left(\int_{[b^{\bsj} \star \bsell, b^{\bsj} \star (\bsell + \boldsymbol{1}))} |f(\bsx)|^2 \rd \bsx \right)^{1/2} \\ & \le & \gamma_\emptyset(J) V_{\bsalpha,\bsgamma}(f).
\end{eqnarray*}

Now assume $\bsr \neq \bszero$ and let $A_{\bsr} = \{\bsa = (a_1,\ldots, a_s) \in \mathbb{N}_0^s: 0 \le a_i < b^{r_i} \mbox{ for } 1 \le i \le s\}$. For some $\bsa \in A_{\bsr}$ let $\bsx \in [b^{\bsj} \star \bsell + b^{\bsj - \bsr} \star \bsa, b^{\bsj} \star \bsell + b^{\bsj-\bsr} \star (\bsa + \bsone))$ and
\begin{eqnarray*}
g_{\bsj,\bsr,\bsell}(\bsx) & = & \sum_{k_1 = 0}^{b^{r_1}-1} \cdots \sum_{k_s = 0}^{b^{r_s}-1} \widehat{f}_{\bsj,\bsk,\bsell} w_{\bsj,\bsk,\bsell}(\bsx) \\ & = & \int_{[b^{\bsj} \star \bsell, b^{\bsj} \star (\bsell + \bsone))} \sum_{k_1 =  0}^{b^{r_1}-1} \cdots \sum_{k_s = 0}^{b^{r_s}-1} f(\bsy) w_{\bsj,\bsk,\bsell}(\bsx) \overline{w_{\bsj,\bsk,\bsell}(\bsy)} \rd \bsy \\ & = & b^{-j_1 - \cdots - j_s} \sum_{k_1 =  0}^{b^{r_1}-1} \cdots \sum_{k_s = 0}^{b^{r_s}-1} \\ && \qquad \int_{[b^{\bsj} \star \bsell, b^{\bsj} \star (\bsell + \bsone))}  f(\bsy) \walb_{\bsk}(b^{-\bsj} \star \bsx - \bsell) \overline{\walb_{\bsk}(b^{-\bsj} \star \bsy - \bsell)} \rd \bsy \\ & = &  b^{-j_1 - \cdots - j_s}  \int_{[b^{\bsj} \star \bsell, b^{\bsj} \star (\bsell + \bsone))} f(\bsy) \sum_{k_1 =  0}^{b^{r_1}-1} \cdots \sum_{k_s = 0}^{b^{r_s}-1}   \walb_{\bsk}(b^{-\bsj} \star (\bsx \ominus \bsy)) \rd \bsy \\ & = &  b^{-j_1 - \cdots - j_s + r_1 + \cdots + r_s}  c_{\bsr,\bsa},
\end{eqnarray*}
where
\begin{equation*}
c_{\bsr,\bsa} = \int_{[b^{\bsj} \star \bsell + b^{\bsj - \bsr} \star \bsa , b^{\bsj} \star \bsell + b^{\bsj-\bsr} \star (\bsa + \bsone))} f(\bsy) \rd \bsy.
\end{equation*}

Let now  $g(\bsx) = 0$ for $\bsx \notin  [b^{\bsj} \star \bsell, b^{\bsj} \star (\bsell + \boldsymbol{1}))$ and otherwise let
\begin{eqnarray*}
g(\bsx) & = & \sum_{k_1 = b^{r_1-1}}^{b^{r_1}-1} \cdots \sum_{k_s =
b^{r_s-1}}^{b^{r_s}-1} \widehat{f}_{\bsj,\bsk,\bsell}
w_{\bsj,\bsk,\bsell}(\bsx) \\ & = &  \sum_{u \subseteq u_{\bsr}} (-1)^{|u|} g_{\bsj,\bsr - (\bsone_u, \bszero_{S \setminus u}), \bsell}(\bsx).
\end{eqnarray*}

Then $g$ is constant on intervals of the form $[b^{\bsj} \star \bsell + b^{\bsj - \bsr} \star \bsa, b^{\bsj} \star \bsell + b^{\bsj-\bsr} \star (\bsa + \bsone))$ and therefore
\begin{eqnarray*}
\sigma_{\bsj,\bsr,\bsell}^2(f) & = & \int_{\mathbb{R}^s} |g(\bsx)|^2
\rd \bsx \\ & = & \int_{[b^{\bsj} \star \bsell, b^{\bsj} \star
(\bsell + \boldsymbol{1}))} |g(\bsx)|^2 \rd \bsx \\ & = & \sum_{\bsa \in A_{\bsr}} \int_{[b^{\bsj} \star \bsell + b^{\bsj-\bsr} \star \bsa, b^{\bsj} \star \bsell + b^{\bsj - \bsr} \star (\bsa + \bsone))} |g(\bsx)|^2 \rd \bsx \\ & = & b^{- (j_1 + \cdots + j_s) + r_1 + \cdots + r_s} \sum_{\bsa \in A_{\bsr}}  \left|\sum_{u \subseteq u_{\bsr}} (-1)^{|u|} c_{\bsr - (\bsone_u, \bszero_{S \setminus u}), (\lfloor \bsa_u/b \rfloor, \bsa_{S \setminus u})} \right|^2,
\end{eqnarray*}
where $(\lfloor \bsa_\uu /b \rfloor, \bsa_{\cS \setminus \uu})$ is the vector whose $i$th component is $\lfloor a_i/b \rfloor$ for $i \in \uu$ and $a_i$ otherwise. Let
\begin{equation*}
d_{\bsr,\bsa} = \int_{[b^{\bsj} \star \bsell + b^{\bsj - \bsr} \star \bsa, b^{\bsj} \star \bsell + b^{\bsj-\bsr} \star (\bsa + \bsone))} |f(\bsy)| \rd \bsy.
\end{equation*}
Then
\begin{eqnarray*}
\sigma_{\bsj,\bsr,\bsell}^2(f) & \le & b^{-(j_1 + \cdots + j_s) + r_1 + \cdots + r_s} 4^{|\uu_{\bsr}|} \sum_{\bsa \in A_{\bsr}} d_{\bsr-\bsone, \lfloor \bsa/b \rfloor}^2 \\ & \le & b^{-(j_1 + \cdots + j_s) + r_1 + \cdots + r_s} 4^{|\uu_{\bsr}|} \\ && \times \sum_{\bsa \in A_{\bsr}} \int_{[b^{\bsj} \star \bsell + b^{\bsj - \bsr + (\bsone_{u_{\bsr}}, \bszero)} \star \lfloor \bsa/b \rfloor, b^{\bsj} \star \bsell + b^{\bsj-\bsr + (\bsone_{\uu_{\bsr}},\bszero)} \star (\lfloor \bsa/b \rfloor + (\bsone_{u_{\bsr}},\bszero) ))} 1 \rd \bsy  \\ && \times \int_{[b^{\bsj} \star \bsell + b^{\bsj - \bsr + (\bsone_{u_{\bsr}},\bszero)} \star \lfloor \bsa/b \rfloor, b^{\bsj} \star \bsell + b^{\bsj-\bsr + (\bsone_{u_{\bsr}}, \bszero)} \star (\lfloor \bsa/b \rfloor + (\bsone_{u_{\bsr}}, \bszero)))} |f(\bsy)|^2 \rd \bsy \\ & \le & b^{|u_{\bsr}|} 4^{|u_{\bsr}|} \sum_{\bsa \in A_{\bsr}} \int_{[b^{\bsj} \star \bsell + b^{\bsj - \bsr + (\bsone_{u_{\bsr}},\bszero) } \star \lfloor \bsa/b \rfloor, b^{\bsj} \star \bsell + b^{\bsj-\bsr + (\bsone_{\uu_{\bsr}},\bszero)} \star (\lfloor \bsa/b \rfloor + (\bsone_{u_{\bsr}},\bszero) ))} |f(\bsy)|^2 \rd \bsy \\ & = &  b^{2|u_{\bsr}|} 4^{|u_{\bsr}|} \int_{[b^{\bsj} \star \bsell, b^{\bsj} \star (\bsell + \bsone))} |f(\bsy)|^2 \rd \bsy \\ & \le & b^{2|u_{\bsr}|} 4^{|u_{\bsr}|} \gamma^2_{\emptyset}(J) V^2_{\bsalpha,\bsgamma}(f).
\end{eqnarray*}
\end{proof}

We analyze now the behavior of the Walsh coefficients in terms of
$\bsalpha, \bsgamma$. Let $J = [b^{\bsj} \star \bsell, b^{\bsj} \star (\bsell + \bsone))$.
\begin{itemize}
\item[(i)] Roughly speaking, the parameter $\alpha_u$ controls how fast $\sigma_{\bsj,\bsr,\bsl}(f)$ decays when the location $J$ (i.e. the support of the Walsh function $w_{\bsj,\bsk,\bsl}$) is fixed but the frequency $\bsr$ increases. This follows from $(i)$ of Lemma~\ref{lem_boundw} because of the factor $b^{-\alpha_{u_{\bsr}}(J) |\bsr|_1}$; note that the dependence of $\alpha_{u_{\bsr}}(J)$ on the location is limited, since $\alpha_{u_{\bsr}}(J) \le 1$.
\item[(ii)] The function $\gamma_{\emptyset}$ controls how fast $\sigma_{\bsj,\bsr,\bsl}(f)$ decays when the frequency $\bsk$ (or $\bsr$) is fixed, but the location $J$ changes. This is modelled through the behavior of $\gamma_{\emptyset,\bsj,\bsl}$ and follows from $(ii)$ of Lemma~\ref{lem_boundw}.
\item[(iii)] Let $\emptyset \neq u \subseteq S$ be fixed and let $\bsr = (\bsr_u,\bszero)$ with $\bsr_u \in \mathbb{N}^{|u|}$. Now consider a change of location and frequency in the coordinates in $u$ simultaneously. From Lemma~\ref{lem_boundw} $(i)$ it follows that $\sigma_{\bsj,\bsr,\bsl}(f)$ decays with $b^{-\alpha_u |\bsr|_1} \gamma_{u}(J)$. Hence if $\gamma_{u}(J)$ decreases as $J$ moves towards infinity, then the diagonal elements, where the frequency and location increase simultaneously, decay faster than if just the frequency increases.
\end{itemize}

We prove a result concerning the convergence of the Walsh series for functions $f$ with $V_{\bsalpha,\bsgamma}(f) < \infty$. The result is analogous to \cite[Theorem~XVI]{Fine}.

\begin{theorem}\label{thm_conv}
Let $\bsalpha$ be local smoothness parameters such that
\begin{equation*}
\inf_{J \in \Pcal_{\mathbb{R}^s}} \alpha_u(J) > 1/2
\end{equation*}
for all $\emptyset \neq u \subseteq S$ and let $\bsgamma$ be
local weight parameters. Let $f \in H_{\bsalpha,\bsgamma}$. Let $\boldsymbol{D} \subset \mathbb{Z}^{2s}$ be such that $\{[b^{\bsj} \star \bsl, b^{\bsj} \star (\bsl + \boldsymbol{1})): (\bsj,\bsl) \in \boldsymbol{D}\}$ is a partition of $\mathbb{R}^s$ and consider the set of tiles
\begin{equation*}
\{T_{\bsj,\bsk,\bsl}: \lfloor b^{r_i-1} \rfloor \le k_i < b^{r_i} \mbox{ for } 1 \le i \le s \mbox{ for some } (\bsj,\bsl) \in \boldsymbol{D}, \bsr \in \mathbb{N}_0^s\},
\end{equation*}
which forms a partition of $(\mathbb{R} \times \mathbb{R}_0^+)^s$. Further assume that for all $u \subseteq S$ we have
\begin{equation*}
\sum_{(\bsj,\bsl) \in \boldsymbol{D}} b^{\alpha_{u}(J) \sum_{i \in u} j_i - \sum_{i \notin u} j_i/2} \gamma_{u}(J) < \infty,
\end{equation*}
where $J = [b^{\bsj} \star \bsell, b^{\bsj} \star (\bsell + \bsone))$.

Then the Walsh series
\begin{equation*}
\sum_{(\bsj,\bsl) \in \boldsymbol{D}} \sum_{\bsk \in \mathbb{N}_0^s} \widehat{f}_{\bsj,\bsk,\bsl} w_{\bsj,\bsk,\bsl}(\bsx)
\end{equation*}
converges absolutely and we have
\begin{equation*}
f(\bsx) = \sum_{(\bsj,\bsl) \in \boldsymbol{D}} \sum_{\bsk \in \mathbb{N}_0^s} \widehat{f}_{\bsj,\bsk,\bsl} w_{\bsj,\bsk,\bsl}(\bsx)
\end{equation*}
for all $\bsx \in \mathbb{R}^s$.
\end{theorem}

\begin{proof}
From Lemma~\ref{lem_boundw} and the Cauchy-Schwarz inequality we obtain
\begin{eqnarray*}
\lefteqn{\sum_{k_1 = \lfloor b^{r_1-1} \rfloor}^{b^{r_1}-1} \cdots \sum_{k_s = \lfloor b^{r_s-1} \rfloor}^{b^{r_s}-1} |\widehat{f}_{\bsj,\bsk,\bsl}| } \\ & \le & \sigma_{\bsj,\bsr,\bsl} \left(\sum_{k_1 = \lfloor b^{r_1-1} \rfloor}^{b^{r_1}-1} \cdots \sum_{k_s = \lfloor b^{r_s-1} \rfloor}^{b^{r_s}-1} 1 \right)^{1/2} \\ & \le & (b-1)^{\alpha_{u_{\bsr}}(J) |u_{\bsr}|}
b^{(1/2-\alpha_{u_{\bsr}}(J)) |\bsr|_1} b^{\alpha_{u_{\bsr}}(J) \sum_{i \in u_{\bsr}} j_i - \sum_{i \notin u_{\bsr}} j_i/2} \gamma_{u_{\bsr}}(J) V_{\bsalpha,\bsgamma}(f).
\end{eqnarray*}
By the assumptions of the theorem, the last expression is summable and hence the Walsh series is absolutely convergent.

Since the Walsh series converges absolutely, its partial sums form a Cauchy sequence. In \cite{SWS} (or see also \cite[Appendix~A.3]{DP09}) it was shown that the sums
\begin{equation*}
\sum_{(\bsj,\bsl) \in \boldsymbol{D}} \sum_{\bsk \in \prod_{i=1}^s \{0,\ldots, b^{r_i}-1 \}} \widehat{f}_{\bsj,\bsk,\bsl} w_{\bsj,\bsk,\bsl}(\bsx)
\end{equation*}
converge to $f(\bsx)$ as $r_1,\ldots, r_s \rightarrow \infty$. Hence the convergence of the Walsh series to $f$ follows.
\end{proof}

We assume throughout the paper that the assumptions of Theorem~\ref{thm_conv} are satisfied.

\section{Numerical integration}\label{sec_num_int}

In this section we study the worst-case error for numerical integration in the unit ball of $H_{\bsalpha,\bsgamma}$, that is,
\begin{equation*}
e(H_{\bsalpha,\bsgamma},Q_{P,\Lambda}) = \sup_{f \in
H_{\bsalpha,\bsgamma}, V_{\alpha,\bsgamma}(f) \le 1}
\left|Q_{P,\Lambda}(f) - \int_{\mathbb{R}^s} f(\bsx) \, {\rm d} \bsx
\right|.
\end{equation*}
Here, the quadrature formula is of the form
\begin{equation*}
Q_{P,\Lambda}(f) = \sum_{n=0}^{N-1} \lambda_n f(\bsx_n),
\end{equation*}
where $\lambda_0,\ldots, \lambda_{N-1}$ are positive weights and $\bsx_0,\ldots, \bsx_{N-1} \in \mathbb{R}^s$ are quadrature points. The guiding principle for choosing the weights is the idea to have equal weight quadrature rules locally on elementary intervals of $\mathbb{R}^s$ which yield a small integration error. Let $J$ be an elementary interval from a given a partition of $\mathbb{R}^s$ into elementary intervals. Let $N_J$ be the number of quadrature points in $J$, then the weight corresponding to those quadrature points is given by $\mathrm{Vol}(J) N_J^{-1}$.

We consider tilings of the phase plane which allow us to use Lemma~\ref{lem_boundw}. For a location fixed by $\bsj,\bsl \in \mathbb{Z}^s$ we include all frequencies $\bsk \in \mathbb{N}_0^s$. Let $\boldsymbol{D} \subset \mathbb{Z}^{2s}$ be such that the intervals $[b^{\bsj} \star \bsl, b^{\bsj} \star (\bsl + \boldsymbol{1}))$ for $(\bsj,\bsl) \in \boldsymbol{D}$ form a partition of $\mathbb{R}^s$. The set
\begin{equation*}
\boldsymbol{B} = \{(\bsj,\bsk,\bsl): (\bsj,\bsl) \in \boldsymbol{D}, \bsk \in \mathbb{N}_0^s\}
\end{equation*}
defines a disjoint tiling $\{T_{\bsj,\bsk,\bsl}: (\bsj,\bsk,\bsl) \in \boldsymbol{B}\}$,  which covers $(\mathbb{R} \times \mathbb{R}_0^+)^s$, that is,
\begin{itemize}
\item $T_{\bsj,\bsk,\bsl} \cap T_{\bsj',\bsk',\bsl'} = \emptyset$ for all $(\bsj,\bsk,\bsl), (\bsj',\bsk',\bsl') \in \boldsymbol{B}$ with $(\bsj,\bsk,\bsl) \neq (\bsj',\bsk',\bsl')$, and
\item $\bigcup_{(\bsj,\bsk,\bsl) \in \boldsymbol{B}} T_{\bsj,\bsk,\bsl} = (\mathbb{R} \times \mathbb{R}_0^+)^s$.
\end{itemize}

This ensures that the corresponding system
\begin{equation*}
\{w_{\bsj,\bsk,\bsl}: (\bsj,\bsk,\bsl) \in \boldsymbol{B}\}
\end{equation*}
is a complete orthonormal system of $L_2(\mathbb{R}^s)$.

For $(\bsj,\bsr,\bsl) \in \boldsymbol{B}$ let
\begin{equation*}
\delta_{\bsj,\bsr,\bsl} = \left( \sum_{k_1= \lfloor b^{r_1-1} \rfloor}^{b^{r_1}-1} \cdots \sum_{k_s = \lfloor b^{r_s-1}\rfloor}^{b^{r_s}-1} \left| \sum_{n=0}^{N-1} \lambda_n w_{\bsj,\bsk,\bsl}(\bsx_n) - \int_{\mathbb{R}^s} w_{\bsj,\bsk,\bsl}(\bsx) \, {\rm d} \bsx \right|^2\right)^{1/2}.
\end{equation*}

\begin{theorem}\label{thm_boundinterror}
Let $\bsalpha$ be local smoothness parameters and $\bsgamma$ be
local weight parameters. Let $Q_{P, \Lambda}$ be a
quadrature rule and let $\boldsymbol{B}$, $\delta_{\bsj,\bsr,\bsl}$,
and $\sigma_{\bsj,\bsr,\bsl}(f)$ be defined as above. For $\bsr \in
\mathbb{N}_0^s$ let $u_{\bsr} = \{i \in S: r_i \neq 0\}$. Then we
have
\begin{eqnarray*}
\lefteqn{ e(H_{\bsalpha,\bsgamma},Q_{P,\Lambda})
\le  \sum_{(\bsj,\bsr,\bsl) \in
\boldsymbol{B}} \sigma_{\bsj,\bsr,\bsl}(f) \delta_{\bsj,\bsr,\bsl} }
\\  & \le & \sum_{(\bsj,\bsl) \in \boldsymbol{D}} \gamma_\emptyset(J) \delta_{\bsj,\bszero, \bsell} \\ && +  \sum_{(\bsj,\bsl) \in \boldsymbol{D}}  \sum_{\bsr \in \mathbb{N}_0^s \setminus \{\bszero\}} (b-1)^{\alpha_{u_{\bsr}}(J) |u_{\bsr}|}
b^{-\alpha_{u_{\bsr}}(J) |\bsr|_1} b^{\alpha_{u_{\bsr}}(J) \sum_{i \in u_{\bsr}} j_i} b^{- \sum_{i \notin u_{\bsr}} j_i/2} \gamma_{u_{\bsr}}(J) \delta_{\bsj,\bsr,\bsl}.
\end{eqnarray*}
\end{theorem}

\begin{proof}
Let
\begin{equation*}
f(\bsx) = \sum_{(\bsj,\bsk,\bsl) \in \boldsymbol{B}} \widehat{f}_{\bsj,\bsk,\bsl} w_{\bsj,\bsk,\bsl}(\bsx).
\end{equation*}

Consider the quadrature formula
\begin{equation*}
Q_{P,\Lambda}(f) = \sum_{n=0}^{N-1} \lambda_n f(\bsx_n).
\end{equation*}
We have
\begin{eqnarray*}
\lefteqn{ Q_{P,\Lambda}(f) - \int_{\mathbb{R}^s}  f(\bsx) \,{\rm d}
\bsx } \\ & = & \sum_{(\bsj,\bsk,\bsl) \in \boldsymbol{B}}
\widehat{f}_{\bsj,\bsk,\bsl} \left[ \sum_{n=0}^{N-1} \lambda_n
w_{\bsj,\bsk,\bsl}(\bsx_n) - \int_{\mathbb{R}^s}
w_{\bsj,\bsk,\bsl}(\bsx) \, {\rm d} \bsx \right] \\ & = &
\sum_{(\bsj,\bsr,\bsl) \in \boldsymbol{B}} \sum_{k_1 = \lfloor
b^{r_1-1} \rfloor}^{b^{r_1}-1} \cdots \sum_{k_s = \lfloor b^{r_s-1}
\rfloor}^{b^{r_s}-1} \widehat{f}_{\bsj,\bsk,\bsl} \left[
\sum_{n=0}^{N-1} \lambda_n w_{\bsj,\bsk,\bsl}(\bsx_n) -
\int_{\mathbb{R}^s} w_{\bsj,\bsk,\bsl}(\bsx) \, {\rm d} \bsx
\right].
\end{eqnarray*}
Using the Cauchy-Schwarz inequality we obtain
\begin{eqnarray*}
\left| Q_{P,\Lambda}(f) - \int_{\mathbb{R}^s}  f(\bsx) \,{\rm d} \bsx \right| & \le & \sum_{(\bsj,\bsr,\bsl) \in \boldsymbol{B}} \sigma_{\bsj,\bsr,\bsl}(f) \delta_{\bsj,\bsr,\bsl}.
\end{eqnarray*}
The second bound follows by using Lemma~\ref{lem_boundw} $(i)$.
\end{proof}

In the bound on the integration error in Theorem~\ref{thm_boundinterror} the only factor which depends on the quadrature points $P = \{\bsx_0,\ldots, \bsx_{N-1}\}$ and weights $\Lambda = \{\lambda_0,\ldots, \lambda_{N-1}\}$ is $\delta_{\bsj,\bsr,\bsl}$. In the following lemma we show that a certain choice of weights will guarantee that many $\delta_{\bsj,\bszero,\bsl}$ are zero.

\begin{lemma}\label{lem_weights}
Let $P =  \{\bsx_0,\ldots, \bsx_{N-1}\} \subset \mathbb{R}^s$ be a set of quadrature points.  For $(\bsj,\bsl) \in\boldsymbol{D}$ let $$N_{\bsj,\bsl} = \{0 \le n < N: \bsx_n \in [b^{\bsj} \star \bsl, b^{\bsj} \star (\bsl+\boldsymbol{1})) \}.$$ Then we have $\bigcup_{(\bsj,\bsl) \in \boldsymbol{D}} N_{\bsj,\bsl} = \{0, \ldots, N-1\}$ and $N_{\bsj,\bsl} \cap N_{\bsj',\bsl'} = \emptyset$ for $(\bsj,\bsl) \neq (\bsj',\bsl')$. For $n \in N_{\bsj,\bsl}$ let
\begin{equation}\label{def_weights}
\lambda_n = \lambda_{\bsj,\bsl} = \frac{b^{j_1+\cdots + j_s}}{|N_{\bsj,\bsl}|}.
\end{equation}
Then for all $(\bsj,\bsr,\bsl) \in \boldsymbol{B}$ we have:
\begin{itemize}
\item[(i)] If $|N_{\bsj,\bsl}| > 0$, then $\delta_{\bsj,\bszero,\bsl} = 0$;
\item[(ii)] If $|N_{\bsj,\bsl}| = 0$, then $\delta_{\bsj,\bszero, \bsl} = b^{(j_1 + \cdots + j_s)/2}$;
\item[(iii)] If $|N_{\bsj,\bsl}| = 0$ and $\bsr \in \mathbb{N}_0^s \setminus\{\bszero\}$, then $\delta_{\bsj,\bsr,\bsl} =  0$.
\end{itemize}
\end{lemma}

\begin{proof}
Since the intervals $[b^{\bsj} \star \bsl, b^{\bsj} \star (\bsl + \boldsymbol{1}))$, for  $(\bsj, \bsl) \in \boldsymbol{D}$, form a partition of $\mathbb{R}^s$, the sets $N_{\bsj,\bsl}$ form a partition of $\{0, \ldots, N-1\}$.

We have
\begin{equation*}
\int_{-\infty}^\infty w_{j,k,l}(x) \,{\rm d} x = \left\{\begin{array}{ll} b^{j/2} & \mbox{if } k = 0, \\ 0 & \mbox{otherwise}. \end{array} \right.
\end{equation*}
Hence, for $\bsr = \bszero$ and $|N_{\bsj,\bsl}| > 0$ we have
\begin{equation}\label{eq_sumr0delta}
\delta_{\bsj,\bszero,\bsl} = \left|\sum_{n=0}^{N-1} \lambda_n w_{\bsj,\bszero,\bsl}(\bsx_n) - b^{(j_1 + \cdots + j_s)/2} \right|  = \left|\sum_{n \in N_{\bsj,\bsl}} \frac{b^{(j_1+\cdots + j_s)/2}}{|N_{\bsj,\bsl}|} - b^{(j_1 + \cdots + j_s)/2} \right| = 0.
\end{equation}
This shows $(i)$.

If $|N_{\bsj,\bsl}| = 0$, then the sum $\sum_{n=0}^{N-1} \lambda_n w_{\bsj,\bszero,\bsl}(\bsx_n) = 0$. Hence \eqref{eq_sumr0delta} implies $(ii)$.

For $\bsr \neq \bszero$ and $|N_{\bsj,\bsl}| = 0$ we have
\begin{eqnarray*}
\delta_{\bsj,\bsr,\bsl}^2 & = & \sum_{k_1= \lfloor b^{r_1-1} \rfloor}^{b^{r_1}-1} \cdots \sum_{k_s = \lfloor b^{r_s-1} \rfloor}^{b^{r_s}-1} \left| \sum_{n=0}^{N-1} \lambda_n w_{\bsj,\bsk,\bsl}(\bsx_n)\right|^2 \\ & = & \sum_{n,n'=0}^{N-1} \lambda_n \overline{\lambda_{n'}} \sum_{k_1 = \lfloor b^{r_1-1} \rfloor}^{b^{r_1}-1} \cdots \sum_{k_s = \lfloor b^{r_s-1}\rfloor}^{b^{r_s}-1} w_{\bsj,\bsk,\bsl}(\bsx_n) \overline{w_{\bsj,\bsk,\bsl}(\bsx_{n'})} \\ & = & 0,
\end{eqnarray*}
since no point of $P$ lies in the support of $w_{\bsj,\bsk,\bsl}$, which implies $(iii)$.
\end{proof}

For the remainder of the paper we assume the weights $\Lambda$ are
given by \eqref{def_weights}. In this case we have for $\bsr \neq
\bszero$ that
\begin{eqnarray}\label{eq_deltar}
\delta_{\bsj,\bsr,\bsl}^2 & = & b^{-j_1 - \cdots - j_s}
|\lambda_{\bsj,\bsl}|^2 \sum_{k_1  = \lfloor b^{r_1-1}
\rfloor}^{b^{r_1}-1} \cdots \sum_{k_s = \lfloor b^{r_s-1}
\rfloor}^{b^{r_s}-1} \sum_{n, n' \in N_{\bsj,\bsl}}   {\rm
wal}_{\bsk}((b^{-\bsj} \star \bsx_n - \bsl) \ominus (b^{-\bsj}
\bsx_{n'} - \bsl)) \nonumber \\ & = & b^{j_1+\cdots + j_s} \hspace{-2mm} \sum_{k_1= \lfloor
b^{r_1-1}\rfloor }^{b^{r_1}-1} \hspace{-2mm} \cdots \hspace{-2mm} \sum_{k_s = \lfloor b^{r_1-1}
\rfloor}^{b^{r_s}-1}  \frac{1}{N_{\bsj,\bsl}^2} \sum_{n,n' \in
N_{\bsj,\bsl}}  {\rm wal}_{\bsk}((b^{-\bsj} \star \bsx_n - \bsl)
\ominus (b^{-\bsj} \star \bsx_{n'} - \bsl)).
\end{eqnarray}

Using the bound in Theorem~\ref{thm_boundinterror} and Lemma~\ref{lem_weights} we obtain the following result.

\begin{lemma}\label{lem_bound2}
Let $\bsalpha$ be local smoothness parameters and $\bsgamma$ be
local weight parameters. Let $P = \{\bsx_0,\ldots,
\bsx_{N-1}\}$ be a set of quadrature points and $\Lambda =
\{\lambda_0,\ldots, \lambda_{N-1}\}$ be given by
\eqref{def_weights}. Let $\delta_{\bsj,\bsr,\bsl}$ be defined as
above. Then we have
\begin{eqnarray*}
\lefteqn{ e(H_{\bsalpha,\bsgamma}, Q_{P,\Lambda})  \le  \sum_{(\bsj,\bsl)
\in \boldsymbol{D}, |N_{\bsj,\bsl}| = 0} b^{(j_1+\cdots + j_s)/2}
\gamma_{\emptyset}(J) } \\ &&  + \sum_{\emptyset \neq u
\subseteq S}   \sum_{(\bsj,\bsl) \in \boldsymbol{D}, |N_{\bsj,\bsl}| > 0} (b-1)^{\alpha_{u}(J) |u|} b^{\alpha_{u}(J) \sum_{i\in u} j_i - \sum_{i \notin u} j_i/2}  \gamma_{u}(J) \sum_{\bsr_u \in \mathbb{N}^{|u|}} b^{-\alpha_{u}(J)
|\bsr|_1} \delta_{\bsj,(\bsr_u,\bszero),\bsl}.
\end{eqnarray*}
\end{lemma}
The error bound in the lemma above consists of two parts. The
first sum arises from the fact that the points are only spread over
a finite area and hence one obtains a truncation error, whereas the
second term arises from the approximation error of the integral over
the region where there are points.

In the following we define translated and dilated digital nets in
arbitrary elementary intervals in $\mathbb{R}^s$.

\begin{definition}
Let $\bsj,\bsl \in \mathbb{Z}^s$ be given and let $P =
\{\bsx_0,\ldots, \bsx_{b^m-1}\}$ be a point set in $I = [b^{\bsj}
\star \bsl, b^{\bsj} \star (\bsl + \boldsymbol{1}))$. Let
$T:[b^{\bsj} \star \bsl, b^{\bsj} \star (\bsl + \boldsymbol{1}))
\rightarrow [0,1]^s$ be given by
\begin{equation*}
T(\bsx) = b^{-\bsj} \star \bsx - \bsl.
\end{equation*}
Then we call $P$ a digitally shifted digital $(t,m,s)$-net over
$\mathbb{Z}_b$ in $I$, if the translated and dilated point set
$\{T(\bsx_0), \ldots, T(\bsx_{b^m-1})\} \subset [0,1)^s$ is a
digitally shifted digital net over $\mathbb{Z}_b$.
\end{definition}

In the following we analyze the integration error when for each $(\bsj,\bsl) \in \boldsymbol{D}$ with $N_{\bsj,\bsl} > 0$ the points in $[b^{\bsj} \star \bsl, b^{\bsj} \star (\bsl + \boldsymbol{1}))$ form a digitally shifted digital net.

\begin{lemma}
Let $(\bsj,\bsl) \in  \boldsymbol{D}$ be given. Let $P =
\{\bsx_0,\ldots, \bsx_{b^m-1}\} \subset [b^{\bsj} \star \bsl,
b^{\bsj} \star (\bsl + \boldsymbol{1}))$ be a digitally shifted
digital $(t,m,s)$-net over $\mathbb{Z}_b$ in $[b^{\bsj} \star \bsl,
b^{\bsj} \star (\bsl + \boldsymbol{1}))$. Let $\bsr \in
\mathbb{N}_0^s \setminus \{\bszero\}$. Then
\begin{equation*}
\delta_{\bsj,\bsr,\bsl} \le \left\{\begin{array}{ll} 0 & \mbox{if }
|\bsr|_1 \le m - t, \\ (1-1/b)^{|u_{\bsr}|/2} b^{(j_1+\cdots +j_s)/2} b^{(|\bsr|_1 - m + t)/2}
& \mbox{if } |\bsr|_1 > m-t.
\end{array} \right.
\end{equation*}
\end{lemma}

\begin{proof}
Let $\bsy_n = T(\bsx_n)$ for $0 \le n < b^m$, then $P_T =
\{\bsy_0,\ldots, \bsy_{b^m-1}\}$ is a digitally shifted digital
$(t,m,s)$-net over $\mathbb{Z}_b$. We have
\begin{eqnarray*}
\delta_{\bsj,\bsr,\bsl}^2 & = & b^{j_1+\cdots + j_s} \sum_{k_1= \lfloor b^{r_1-1}\rfloor
}^{b^{r_1}-1} \cdots \sum_{k_s = \lfloor b^{r_1-1}
\rfloor}^{b^{r_s}-1}  \frac{1}{b^{2m}} \sum_{n,n' = 0}^{b^m-1} {\rm
wal}_{\bsk}((b^{-\bsj} \star \bsx_n - \bsl) \ominus (b^{-\bsj} \star
\bsx_{n'} - \bsl)) \\ & = & b^{j_1+\cdots + j_s} \sum_{k_1= \lfloor b^{r_1-1}\rfloor
}^{b^{r_1}-1} \cdots \sum_{k_s = \lfloor b^{r_1-1}
\rfloor}^{b^{r_s}-1}  \frac{1}{b^{2m}} \sum_{n,n' = 0}^{b^m-1} {\rm
wal}_{\bsk}(\bsy_n \ominus \bsy_{n'}).
\end{eqnarray*}
Let $C_1,\ldots, C_s \in \mathbb{Z}_b^{m \times m}$ be the
generating matrices of $P_T$ and let $$\mathcal{D} = \{\bsk \in
\mathbb{N}_0^s: C_1 \vec{k}_1 + \cdots + C_s \vec{k}_s \equiv
\vec{0} \in \mathbb{Z}_b^m\}$$ denote the dual net. Then
\begin{equation*}
\frac{1}{b^{2m}} \sum_{n,n'=0}^{b^m-1} {\rm wal}_{\bsk}(\bsy_n
\ominus \bsy_{n'}) = \left\{\begin{array}{ll} 1 & \mbox{if } \bsk
\in \mathcal{D}, \\ 0 & \mbox{otherwise}.
\end{array} \right.
\end{equation*}
Hence it follows that $\delta_{\bsj,\bsr,\bsl} = 0$ if $|\bsr|_1 \le
m-t$.

If $|\bsr|_1 > m-t$, then, as in the proof of \cite[Lemma~7]{CDP},
it follows that
\begin{equation*}
\delta_{\bsj,\bsr,\bsl}^2 \le (1-1/b)^{|u_{\bsr}|} b^{j_1+\cdots + j_s} b^{|\bsr|_1 - m +
t},
\end{equation*}
which implies the result.
\end{proof}

\begin{lemma}\label{lem_dignetsdelta}
Let $(\bsj,\bsl) \in  \boldsymbol{D}$ be given. Let $P =
\{\bsx_0,\ldots, \bsx_{b^m-1}\} \subset [b^{\bsj} \star \bsl,
b^{\bsj} \star (\bsl + \boldsymbol{1}))$ be a digitally shifted
digital $(t,m,s)$-net over $\mathbb{Z}_b$ in $[b^{\bsj} \star \bsl,
b^{\bsj} \star (\bsl + \boldsymbol{1}))$. Let $1/2 < \alpha \le 1$
and let $\emptyset \neq u \subseteq S$. Then
\begin{equation*}
\sum_{\bsr_u \in \mathbb{N}^{|u|}} b^{-\alpha |\bsr_u|_1}
\delta_{\bsj,(\bsr_u,\bszero),\bsl} \le (b-1)^{-|u|/2} b^{|u|/2}
b^{1/2-\alpha} b^{(j_1+\cdots + j_s)/2} b^{-\alpha (m-t)} \binom{m-t + |u|}{|u|-1}.
\end{equation*}
\end{lemma}

\begin{proof}
We have
\begin{eqnarray*}
\sum_{\bsr_u \in \mathbb{N}^{|u|}} b^{-\alpha |\bsr_u|_1}
\delta_{\bsj,(\bsr_u,\bszero),\bsl} & \le & (1-1/b)^{|u|/2} b^{(j_1+\cdots + j_s)/2}
b^{(-m+t)/2} \sum_{\bsr_u \in \mathbb{N}^{|u|}, |\bsr_u|_1 > m-t}
b^{(1/2-\alpha) |\bsr_u|_1} \\ & = & (1-1/b)^{|u|/2}b^{(j_1+\cdots + j_s)/2}  b^{(-m+t)/2}
\sum_{l = m-t+1}^\infty b^{(1/2-\alpha)l} \sum_{\bsr_u \in
\mathbb{N}^{|u|}, |\bsr_u|_1 = l} 1 \\ & \le & (1-1/b)^{|u|/2} b^{(j_1+\cdots + j_s)/2}
b^{(-m+t)/2} \sum_{l = m-t+1}^\infty b^{(1/2-\alpha)l} \binom{l+|u|-1}{|u|-1} \\ & \le & (b-1)^{-|u|/2} b^{|u|/2} b^{1/2-\alpha} b^{(j_1+\cdots + j_s)/2}
b^{-\alpha (m-t)} \binom{m-t + |u|}{|u|-1},
\end{eqnarray*}
where the last inequality follows from \cite[Lemma~6]{CDP}.
\end{proof}

The following theorem follows from Lemmas~\ref{lem_bound2} and
\ref{lem_dignetsdelta}.

\begin{theorem}\label{thm_errorbound}
Let $\bsalpha$ be local smoothness parameters and $\bsgamma$ be local weight parameters. Let a set of quadrature points $P = \{\bsx_0,\ldots, \bsx_{N-1}\}$ be given such that for each $(\bsj, \bsl) \in
\boldsymbol{D}$ with $|N_{\bsj,\bsl}| > 0$ the point sets $\{\bsx_n:
n \in N_{\bsj,\bsl}\}$ are digitally shifted digital
$(t_{\bsj,\bsl},m_{\bsj,\bsl}, s)$-nets over $\mathbb{Z}_b$. Let
$\Lambda = \{\lambda_0,\ldots, \lambda_{N-1}\}$ be given by
\eqref{def_weights}. Then we have
\begin{eqnarray*}
\lefteqn{ e(H_{\alpha,\bsgamma}, Q_{P,\Lambda}) \le   \sum_{(\bsj,\bsl) \in
\boldsymbol{D}, |N_{\bsj,\bsl}| = 0} b^{(j_1+\cdots + j_s)/2}
\gamma_{\emptyset}(J) } \\ &&  + \sum_{\emptyset \neq u
\subseteq S} \sum_{(\bsj,\bsl) \in \boldsymbol{D}, |N_{\bsj,\bsl}|
> 0}  \gamma_{u}(J) C_{\alpha_{u}(J),|u|, J}
b^{-\alpha_{u}(J) (m_{\bsj,\bsl} - t_{\bsj,\bsl})}
\binom{m_{\bsj,\bsl} - t_{\bsj,\bsl} + |u|}{|u|-1},
\end{eqnarray*}
where $J= [b^{\bsj}\star \bsl, b^{\bsj} \star (\bsl +
\boldsymbol{1}))$ and
\begin{equation*}
C_{\alpha_{u}(J),|u|, J} =  (b-1)^{(\alpha_{u}(J) -1/2) |u|} b^{|u|/2}
b^{1/2-\alpha_{|u|}(J)} b^{(\alpha_{u}(J) + 1/2) \sum_{i \in u} j_i}.
\end{equation*}
\end{theorem}


The last theorem lends itself to the following strategy. Use randomized digital nets in regions $J=[b^{\bsj}\star \bsl, b^{\bsj} \star (\bsl +
\boldsymbol{1}))$ where $\gamma_{u}(J)$ is `large'. In regions $J$ where $\gamma_{u}(J)$ is
`small' use less (or no) quadrature points (or redefine the regions so that the regions themselves become larger). Use variance
estimators of the variances of the local integration errors and
increase the number of points where this variance is largest so that
the local integration errors in each region are of approximately
equal size. Hence the values of $\gamma_{u}(J)$ do not have
to be known, instead one adjusts the quadrature points adaptively
using local variance estimators. As the number of quadrature points
increases spread out to cover a larger and larger area.

A problem that can occur is that, if the number of dimensions is
large, there are too many subcubes to consider. For example dividing
a region into two parts in each coordinate for a problem where the
dimension $s = 100$ yields $2^{100} \approx 10^{30}$ subcubes. This
is infeasible. To avoid this problem one can instead only divide the
most important coordinates into smaller intervals, leaving the
majority of the coordinates without any subdivisions.

Another problem that can occur with this method is that one
overlooks an important area where no quadrature points are used since the weight functions are not known in practice. In this case the number of quadrature points can
increase without decreasing the error.

A further disadvantage of this method is that the one-dimensional projections are not optimal, since the points are chosen independently in each subcube.

If the adaptive subdivision strategy fails, one can use the approach
outlined in the following section (in this case the one-dimensional projections are optimal).

\section{Construction of digital nets}\label{sec_construction}

We now turn to the construction of quadrature points. Let the number
of quadrature points be $b^m$, where $b\ge 2$ is a prime number and
$m \ge 1$ is an integer. First, one constructs the one-dimensional
projections for each coordinate and labels the points from $0$ to
$b^m-1$, as illustrated in Figure~\ref{fig1}. Then one uses a
digital net in $[0,1)^s$ with points $\bsx_0,\ldots, \bsx_{b^m-1}$,
where $\bsx_n=(x_{n,1},\ldots, x_{n,s})$ and maps it to
$\mathbb{R}^s$ using the labels for each one-dimensional projection.
That is, the point $x_{n,i}$ is replaced by the point on
$\mathbb{R}$ with label $x_{n,i} b^m$ (note that for digital nets,
$x_{n,i}= kb^{-m}$ for some nonnegative integer $k$). See
Figure~\ref{fig1} for an illustration. We present the details of
this procedure in the following.

\begin{figure}[ht]
\begin{center}
\includegraphics[scale=0.5]{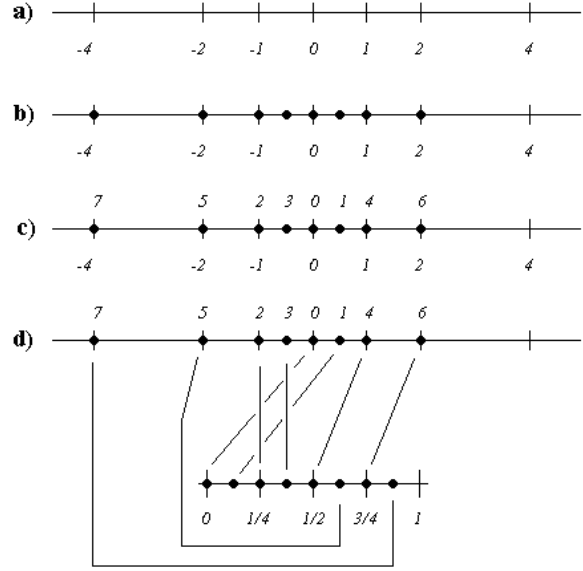}
\caption{\textit{An illustration of the procedure in dimension one
in base $b=2$. a) Decide on a partitioning of $\mathbb{R}$; here we
consider the intervals $[-4,-2), [-2,-1), [-1,0), [0,1), [1,2),
[2,4)$}; b) Decide on how many points should be in each interval and
then place them equally spaced and left-centered; For instance, the
interval $[-1,0)$ contains the points $-1$ and $-1/2$; c) Label the
points, starting with the intervals with the largest number of
points and continue with intervals with a smaller and smaller number
of points; d) Map the points $n/2^m$, $0 \le n < 2^m$ to the point
with label $n$.}
\end{center}\label{fig1}
\end{figure}

Let $b \ge 2$ be an integer. For $1 \le i \le s$ and $1 \le d_i \le \Delta_i$ let $J_{i,d_i} = [b^{j_{i,d_i}} l_{i,d_i}, b^{j_{i,d_i}} (l_{i,d_i} + 1))$ for $j_{i,d_i}, l_{i,d_i} \in \mathbb{Z}$ such that $J_{i,d_i} \cap J_{i,d'_i} =
\emptyset$ for $1 \le d_i < d'_i \le \Delta_i$. Let $0 \le m_{i, \Delta_i} \le m_{i, \Delta_i-1} \le
\cdots \le m_{i,1} \le m$ be integers such that
\begin{equation*}
b^{m_{i,1}} + b^{m_{i,2}} + \cdots + b^{m_{i,\Delta_i}} = b^m.
\end{equation*}
Let $z_{b^{m_{i,1}} + \cdots + b^{m_{i,d_i-1}}}, \ldots, z_{b^{m_{i,1}}+\cdots +
b^{m_{i,d_i-1}} + b^{m_{i,d_i}}-1}$ be $b^{m_{i,d_i}}$ equally spaced points in
$J_{i,d_i}$ such that
\begin{eqnarray*}
z_{b^{m_{i,1}} + \cdots + b^{m_{i,d_i-1}}} & = & b^{j_{i,d_i}} l_{i,d_i}, \\  z_{b^{m_{i,1}}
+ \cdots + b^{m_{i,d_i-1}} + 1} & = & b^{j_{i,d_i}} l_{i,d_i} + b^{j_{i,d_i}-m_{i,d_i}}, \\ &
\vdots & \\ z_{b^{m_{i,1}} + \cdots + b^{m_{i,d_i-1}} + k} &  = & b^{j_{i,d_i}} l_{i,d_i} + k b^{j_{i,d_i} - m_{i,d_i}}   \\ & \vdots &  \\
z_{b^{m_{i,1}} + \cdots + b^{m_{i,d_i}}-1} & = & b^{j_{i,d_i}} (l_{i,d_i}+1) -
b^{j_{i,d_i} - m_{i,d_i}},
\end{eqnarray*}
where for $d_i = 1$ we set $b^{m_{i,1}} + \cdots + b^{m_{i,d_i-1}} = 0$.

Let $C_1,\ldots, C_s \in \mathbb{Z}_b^{m \times m}$ be the
generating matrices of a digital $(t,m,s)$-net over $\mathbb{Z}_b$. See for instance \cite{DP09,Faure,Nied,NiedXing,Sobol} for explicit examples.
For $1 \le i \le s$ and $0 \le n < b^m$, where $n = n_0 + n_1 b +
\cdots + n_{m-1} b^{m-1}$, let
\begin{equation*}
C_i \left(\begin{array}{c} n_0 \\ n_1 \\ \vdots \\ n_{m-1} \end{array} \right) = \left(\begin{array}{c} \eta_{n,i,0} \\ \eta_{n,i,1} \\ \vdots \\
\eta_{n,i,m-1}
\end{array} \right)
\end{equation*}
and set
\begin{equation*}
\eta_{n,i} = \eta_{n,i,0} b^{m-1} + \eta_{n,i,1} b^{m-2} + \cdots +
\eta_{n,i,m-1}.
\end{equation*}
Define
\begin{equation*}
x_{n,i} = z_{\eta_{n,i}} \quad \mbox{and} \quad \bsx_n= (x_{n,1},
\ldots, x_{n,s}).
\end{equation*}

\begin{theorem}
Let $P = \{\bsx_0,\ldots, \bsx_{b^m-1}\}$ be constructed as above
based on a digital $(t,m,s)$-net over $\mathbb{Z}_b$. For $\bsd = (d_1,\ldots, d_s) \in \prod_{i=1}^s \{1,\ldots, \Delta_i\}$ let
\begin{equation*}
m_{\bsd} = m - \sum_{i=1}^s (m-m_{i,d_i}).
\end{equation*}
Then for each $J_{\bsd} = \prod_{i=1}^s J_{i,d_i}$, $1 \le d_i \le \Delta_i$ for $1 \le i
\le s$ with
\begin{equation*}
m_{\bsd} \ge t,
\end{equation*}
the points $P$ in $J_{\bsd}$ form a digitally shifted
digital $(t,m_{\bsd},s)$-net over $\mathbb{Z}_b$ in $J_{\bsd}$. In particular, for all $\bsd \in \prod_{i=1}^s \{1,\ldots, \Delta_i\}$ with $m_{\bsd} \ge t$ it follows that $J_{\bsd}$ contains at least one point of $P$.
\end{theorem}

\begin{proof}
Let $I= [b^{\bsp} \star \bsq, b^{\bsp} \star (\bsq +
\boldsymbol{1})) \subseteq J_{\bsd}$ such that ${\rm Vol}(I) \le
b^{j_{1,d_1} + \cdots + j_{s,d_s} + t - m_{\bsd}}$, that is,
\begin{equation*}
(j_{1,d_1} - p_1) + \cdots + (j_{s,d_s} - p_s) \le m_{\bsd} - t.
\end{equation*}
We need to show that $I$ contains exactly $b^{m_{\bsd} - (j_{1,d_1}
-p_1) - \cdots - (j_{s,d_s} -p_s)}$ points of $P$.

Let $\bsx_n \in I$, then $x_{n,i} = z_{\eta_{n,i}} \in [b^{p_i} q_i,
b^{p_i} (q_i+1))$ for $1 \le i \le s$. Let $k_i = \eta_{n,i} -
b^{m_{i,1}} - \cdots - b^{m_{i,d_i-1}}$, then $z_{\eta_{n,i}} =
b^{j_{i,d_i}} l_{i,d_i} + k_i b^{j_{i,d_i}-m_{i,d_i}}$ and
\begin{equation*}
b^{p_i} q_i \le b^{j_{i,d_i}} l_{i,d_i} + k_i b^{j_{i,d_i}- m_{i,d_i}} <
b^{p_i} (q_i+1)
\end{equation*}
for $1 \le i \le s$. This implies that
\begin{equation*}
b^{p_i-j_{i,d_i} + m_{i,d_i}} q_i - b^{m_{i,d_i}} l_{i,d_i} \le k_i < b^{p_i
- j_{i,d_i} + m_{i,d_i}} (q_i+1) - b^{m_{i,d_i}} l_{i,d_i}
\end{equation*}
and
\begin{eqnarray*}
\lefteqn{ b^{p_i-j_{i,d_i} + m_{i,d_i}} q_i - b^{m_{i,d_i}} l_{i,d_i} + b^{m_{i,1}} +
\cdots + b^{m_{i,d_i-1}} } \\ & \le & \eta_{n,i}  <  b^{p_i - j_{i,d_i} + m_{i,d_i}}
(q_i+1) - b^{m_{i,d_i}} l_{i,d_i} + b^{m_{i,1}} + \cdots + b^{m_{i,d_i-1}}
\end{eqnarray*}
for $1 \le i \le s$. Hence, for $1 \le i \le s$, the digits
\begin{equation*}
\eta_{n,i,0}, \ldots, \eta_{n,i,m-m_{i,d_i} + j_{i,d_i} - p_i-1}
\end{equation*}
are prescribed and the remaining digits can be chosen freely. Let
$a_i = m-m_{i,d_i} + j_{i,d_i} - p_i$,
\begin{equation*}
\zeta_{i,0} = \eta_{n,i,0}, \ldots, \zeta_{i,a_i-1} =
\eta_{n,i,a_i-1}.
\end{equation*}
Let $C_i = (c_{i,1},\ldots, c_{i,m})^\top$ for $1 \le i \le s$.

Then $\bsx_n \in I$ if and only if
\begin{equation}\label{eq_sys}
\left(\begin{array}{c} c_{1,1} \\ \vdots \\ c_{1,a_1} \\ \vdots \\
c_{s,1} \\ \vdots \\ c_{s,a_s}
\end{array} \right)\vec{n} = \left(\begin{array}{c} \zeta_{1,0} \\ \vdots \\ \zeta_{1,a_1-1} \\ \vdots \\ \zeta_{s,0} \\ \vdots \\ \zeta_{s,a_s-} \end{array} \right).
\end{equation}

The vectors $c_{1,1}, \ldots, c_{1,a_1}, \ldots, c_{s,1}, \ldots,
c_{s,a_s}$ are linearly independent since
\begin{equation*}
a_1 + \cdots + a_s = \sum_{i=1}^s (m-m_{i,d_i}) + \sum_{i=1}^s
(j_{i,d_i} - p_i) = m-m_{\bsd} + \sum_{i=1}^s (j_{i,d_i} - p_i) \le m-
t.
\end{equation*}
Hence \eqref{eq_sys} has $b^{m- a_1 - \cdots - a_s}$ solutions and
therefore there are $b^{m-a_1- \cdots - a_s}$ points of $P$ in
$J_{\bsd}$. Since
\begin{equation*}
m- \sum_{i=1}^s a_i = m - \sum_{i=1}^s (m- m_{i,d_i}) - \sum_{i=1}^s
(j_{i,d_i} - p_i) = m_{\bsd} - \sum_{i=1}^s (j_{i,d_i} - p_i),
\end{equation*}
there are exactly $b^{m_{\bsd} - (j_{1,d_1} - p_1) - \cdots - (j_{s,d_s}
- p_s)}$ points of $P$ in $J_{\bsd}$.

Let $N_{\bsd} = \{\vec{n}: 0 \le n < b^m \mbox{ and } \bsx_n \in
J_{\bsd}\}$. From above we see that $N_{\bsd}$ is an affine subspace
of $\mathbb{Z}_b^m$ and hence the set of points in $J_{\bsd}$ form a
digitally shifted digital $(t, m_{\bsd}, s)$-net over $\mathbb{Z}_b$
in $J_{\bsd}$. Hence the result follows.
\end{proof}

Note that if $m_{\bsd} < t$ for some $\bsd \in \prod_{i=1}^s \{1,\ldots, \Delta_i\}$, it is possible that $J_{\bsd}$ does not contain any point of $P$. This follows from the fact that \eqref{eq_sys} may not have any solution.

Further, it is clear from the construction that if the generating matrices $C_1,\ldots, C_s$ of the digital net are nonsingular, then each one-dimensional projection of the quadrature points $\{\bsx_0,\ldots, \bsx_{b^m-1}\}$ yields the points $z_0,\ldots, z_{b^m-1}$.

%

\section{Error bound}\label{sec_error_bound}

We use the construction of Section~\ref{sec_construction} and Theorem~\ref{thm_errorbound} to obtain an upper bound on the integration error.

\begin{theorem}\label{thm_boundnet}
Let $\boldsymbol{D} = \prod_{i=1}^s D_i$, where $D_i \subseteq
\mathbb{Z}^2$ is such that for each $1 \le i \le s$ the collection
of intervals \begin{equation}\label{intervals} \{[b^{j} l,
b^j(l+1)): (j,l) \in D_i\}\end{equation} forms a partition of
$\mathbb{R}$. For $1 \le i\le s$ choose $\Delta_i$ different
intervals from the set \eqref{intervals} represented by $(j_{i,d_i},
l_{i,d_i}) \in D_i$ for $1 \le i \le s$ and $1 \le d_i \le
\Delta_i$, where $(j_{i,d_i}, l_{i,d_i}) \neq (j_{i,d'_i},
l_{i,d'_i})$ for all $d_i \neq d'_i$. Let $E_i = \{(j_{i,d_i},
l_{i,d_i}): 1 \le d_i \le \Delta_i, 1 \le i \le s\}$,
$\boldsymbol{E} = \prod_{i=1}^s E_i$, and
\begin{equation*}
J_{i,d_i} = [b^{j_{i,d_i}} l_{i,d_i}, b^{j_{i,d_i}} (l_{i,d_i} + 1)).
\end{equation*}
Let a digital $(t,m,s)$-net over $\mathbb{Z}_b$ be given. Let $t \le m_{i,\Delta_i} \le \cdots \le m_{i,1} \le m$ be integers such that
\begin{equation*}
b^{m_{i,1}} + \cdots + b^{m_{i,\Delta_i}} = b^m.
\end{equation*}
Let $P = \{\bsx_0,\ldots, \bsx_{b^m-1}\}$ be constructed using the scheme of Section~\ref{sec_construction} based on the digital $(t,m,s)$-net over $\mathbb{Z}_b$. Let $\boldsymbol{F} = \{(\bsj,\bsl) \in \boldsymbol{E}: m_{\bsd} \ge t\}$, where $\bsj = (j_{1,d_1}, \ldots, j_{s,d_s})$, $\bsd = (d_1,\ldots, d_s)$, and $$m_{\bsd} = m - \sum_{i=1}^s (m-m_{i,d_i}).$$  Let
$\Lambda = \{\lambda_0,\ldots, \lambda_{N-1}\}$ be given by
\eqref{def_weights}.

Let $\bsalpha$ be local smoothness parameters and $\bsgamma$ be local weight parameters. Then we have
\begin{eqnarray*}
e(H_{\alpha,\bsgamma}, Q_{P,\Lambda}) & \le &  \sum_{(\bsj,\bsl) \in
\boldsymbol{D} \setminus \boldsymbol{F}} b^{(j_1+\cdots + j_s)/2}
\gamma_{\emptyset}(J) \\ &&  + \sum_{\emptyset \neq u
\subseteq S} \sum_{(\bsj,\bsl) \in \boldsymbol{F}}  \gamma_{u}(J) C_{\alpha,|u|, J}
b^{-\alpha_{u}(J) (m_{\bsd} - t)}
\binom{m_{\bsd} - t + |u|}{|u|-1},
\end{eqnarray*}
where $J= [b^{\bsj}\star \bsl, b^{\bsj}\star (\bsl+\bsone))$ and
\begin{equation*}
C_{\alpha_{u}(J),|u|, J} =  (b-1)^{(\alpha_{u}(J) -1/2) |u|} b^{|u|/2}
b^{1/2-\alpha_{|u|}(J)} b^{(\alpha_{u}(J) + 1/2) \sum_{i \in u} j_i}.
\end{equation*}
\end{theorem}

From the proof of Theorem~\ref{thm_errorbound} it follows that only subcubes $J= [b^{\bsj}\star \bsl, b^{\bsj}\star (\bsl+\bsone))$ are taken into account where $m_{\bsd} \ge t$. Hence Theorem~\ref{thm_boundnet} also applies if one does not use the quadrature points which fall into subcubes $J$ where $m_{\bsd} < t$. The quality parameter $t$ of the original digital $(t,m,s)$-net used in the construction described in Section~\ref{sec_construction} is the same in each subcube $J$ where $m_{\bsd} \ge t$. Hence, in this sense, if one uses a digital $(t,m,s)$-net and places the first $b^{m_{\bsd}}$ points of this net in each of the relevant subcubes, then the quality parameter of the digital nets in each subcube is also $t$, hence nothing is gained by placing the points in each subcube manually instead of using the procedure described in Section~\ref{sec_construction}.

\section{Some Examples}\label{sec_examples_all}

The aim of this section is to show some examples of how Theorem~\ref{thm_boundnet} can be applied in various situations. We do not give a comprehensive overview of how digital nets in $\mathbb{R}^s$ should be constructed for applications. This depends largely on the applications at hand and is left for future work.

\subsection{QMC rules defined on $[0,1]^s$}

The classical case of integration of functions defined on $[0,1]^s$
is included in Theorem~\ref{thm_boundnet} as a special case. It is related to the results on the worst-case setting in \cite{HHY}. Therein Haar functions on $[0,1)$ were used (which just corresponds to a particular tiling of the Walsh phase plane; the decay of the Haar coefficients there was not related to the modulus of continuity).

Assume we are given a partition $\mathcal{P}_{\mathbb{R}^s}$ which includes the unit $[0,1)^s$. Let
\begin{equation*}
\alpha_u(J) = \alpha \quad\mbox{for all } J\in\mathcal{P}_{\mathbb{R}^s} \mbox{ and } \emptyset \neq u \subseteq S,
\end{equation*}
where $1/2 < \alpha \le 1$ and
\begin{equation*}
\gamma_u(J) = \left\{\begin{array}{rr} 1 & \mbox{if } J=[0,1)^s, \\ 0 & \mbox{otherwise}, \end{array}\right.
\end{equation*}
for all $u \subseteq S$. The function space $H_{\bsalpha,\bsgamma}$  is then related to a Besov space of functions $f:[0,1]^s \rightarrow \mathbb{R}$.

Further choose $\Delta_i = 1$ for $1 \le i
\le s$. With these choices one obtains a function space on the
domain $[0,1]^s$ and also qMC rules defined on $[0,1]^s$. The point
set $P = \{\bsx_0,\ldots, \bsx_{b^m-1}\}$ obtained from the
construction in Section~\ref{sec_construction} is the same as the
original digital $(t,m,s)$-net.

In this case we choose $\boldsymbol{D} = \boldsymbol{F} = \{(\bszero, \bszero)\}$. Hence $\boldsymbol{D} \setminus \boldsymbol{F} = \emptyset$. We obtain the following result.

\begin{corollary}
Let $P$ be constructed using the scheme of Section~\ref{sec_construction} based on a digital $(t,m,s)$-net over $\mathbb{Z}_b$. With the choice of parameters described in this subsection we have
\begin{eqnarray*}
e(H_{\alpha,\bsgamma}, Q_{P,\Lambda}) & \le &  b^{-\alpha (m - t)} \sum_{\emptyset \neq u
\subseteq S} C_{\alpha,|u|}
\binom{m - t + |u|}{|u|-1},
\end{eqnarray*}
where
\begin{equation*}
C_{\alpha,|u|} =  (b-1)^{(\alpha -1/2) |u|} b^{1/2-\alpha} b^{|u|/2}.
\end{equation*}
\end{corollary}

\subsection{Rational decay of the weight
parameters}\label{subsec_example2}

Now we consider numerical integration of functions defined on $\mathbb{R}^s$, where the weights decay slowly to $0$ as $|\bsx|_\infty \rightarrow \infty$.

Let the local smoothness parameters $\bsalpha$ be constant, that is
\begin{equation*}
\alpha_u(J) = \alpha \quad\mbox{for all } J \in \mathcal{P}_{\mathbb{R}^s} \mbox{ and for all } \emptyset \neq u \subseteq S,
\end{equation*}
where $1/2 < \alpha \le 1$.

For $J = \prod_{i=1}^s [a_i, b_i)$ let $\gamma_i(J) = (1 + \min(|a_i|, |b_i|)^{2\alpha + 1/2})^{-1}$,
\begin{equation*}
\gamma_u(J) = \prod_{i =1}^s \gamma_i(J)
\end{equation*}
for $\emptyset \neq u \subseteq S$, and let
\begin{equation*}
\gamma_{\emptyset}(J) = \prod_{i=1}^s (1 + \min(|a_i|, |b_i|)^{\alpha+1/2})^{-1}.
\end{equation*}

Let $b = 2$ and $m_l = m-1 - \lceil l/2 \rceil$ for $1 \le l \le 2(m-2)$ and $m_{2m-3} = m_{2m-2} = 1$. Then
\begin{equation*}
\sum_{l=1}^{2m-2} 2^{m_l} = 2[1 + 1 + 2 + 2^2 + \cdots + 2^{m-2}] = 2^m.
\end{equation*}

For $1 \le i \le s$ and $d_i \ge 1$ let
\begin{equation*}
(j_{i,d_i}, l_{i,d_i}) = \left\{\begin{array}{ll} (0,0) & \mbox{for } d_i = 1, \\ (0,-1) & \mbox{for } d_i = 2, \\ ((d_i-3)/2,1) & \mbox{for } d_i \ge 3, d_i \mbox{ odd}, \\ ((d_i-4)/2,-2) & \mbox{for } d_i \ge 4, d_i \mbox{ even}, \end{array} \right.
\end{equation*}
and
\begin{equation*}
J_{i,d_i} = [2^{j_{i,d_i}} l_{i,d_i}, 2^{j_{i,d_i}} (l_{i,d_i} + 1)).
\end{equation*}
The intervals $J_{i,1}, J_{i,2}, J_{i,3}, \ldots$ form a partition of $\mathbb{R}$. Let $$E_i = \{(j_{i,d_i}, l_{i,d_i}): 1 \le d_i \le 2m-2, 1 \le i \le s\}$$ and $\boldsymbol{E} = \prod_{i=1}^s E_i$ and let $$D_i = \{(j_{i,d_i}, l_{i,d_i}): d_i \ge 1, 1 \le i \le s\}$$ and $\boldsymbol{D} = \prod_{i=1}^s D_i$.

Let $P$ be the point set constructed according to Section~\ref{sec_construction} based on a digital $(t,m,s)$-net over $\mathbb{Z}_2$.
\begin{corollary}\label{cor_polydecay}
Let $P$ be constructed using the scheme of Section~\ref{sec_construction} based on a digital $(t,m,s)$-net over $\mathbb{Z}_2$ where $m > t + 3s$ and $s \ge 2$. With the choice of parameters as above we have
\begin{equation*}
e(H_{\bsalpha,\bsgamma}, Q_{P,\Lambda}) \le 2^{-\alpha (m-t)} \binom{m-t-2s}{s}^2 2^{s(3\alpha + 2)} \left(2^{ - \alpha} + 2 (1 + 2^{3/2})^s \right).
\end{equation*}
\end{corollary}

\begin{proof}
We have
\begin{equation*}
\gamma_{i}(J_{i,d_i}) = \left\{\begin{array}{ll} 1 & \mbox{for } d_i = 1,2, \\ (1 + 2^{\alpha j_{i,d_i}})^{-1} & \mbox{for } d_i \ge 3. \end{array} \right.
\end{equation*}
Therefore, for $J_{\bsd} = \prod_{i=1}^s J_{i,d_i}$ we have
\begin{equation*}
\gamma_{u}(J_{\bsd}) \le 2^{- (2\alpha + 1/2) \sum_{i =1}^s j_{i,d_i}}
\end{equation*}
and
\begin{equation*}
\gamma_{\emptyset}(J_{\bsd}) \le 2^{-(\alpha+1/2) \sum_{i=1}^s j_{i,d_i}}.
\end{equation*}

Further we have
\begin{equation*}
\gamma_u(J_{\bsd}) C_{\alpha,|u|} \le 2^{|u|/2} 2^{-\alpha \sum_{i=1}^s j_{i,d_i}}.
\end{equation*}

We have $m_{i,d_i} = m-1 - \lceil d_i/2\rceil$ and hence $m-m_{i,d_i} = 1 + \lceil d_i/2 \rceil$ for $1 \le d_i \le 2(m-2)$. Further we have $m_{i,2m-3} = m_{i,2m-2} = 1$. Thus $m-m_{i,d_i} = \min(1+\lceil d_i/2 \rceil, m-1)$ for $1 \le d_i \le 2m-2$. From $m_{\bsd} \ge t$ it follows that
\begin{equation*}
t \le m_{\bsd} = m- \sum_{i=1}^s (m-m_{i,d_i}) = m-\sum_{i=1}^s \min(1+\lceil d_i/2\rceil, m-1),
\end{equation*}
which implies that
\begin{equation}\label{ineq_ex_rat}
m-t \ge \sum_{i=1}^s \min\left(1 + \lceil d_i/2 \rceil, m-1 \right) = \sum_{i=1}^s \min(j_{i,d_i} + 3,m-1) \le m- t,
\end{equation}
since $m - m_{i,d_i} = 1 + \lceil d_i/2 \rceil = j_{i,d_i} + 3$. Since $m-t\le m$ and $1 + \lceil d_i/2 \rceil \ge 1$ it follows that $j_{i,d_i}+3 \le m-1$ in \eqref{ineq_ex_rat}. Hence
\begin{equation*}
\boldsymbol{F} = \left\{(\bsj, \bsl) \in \boldsymbol{E}: \sum_{i=1}^s j_{i} \le m- t - 3 s \right\}.
\end{equation*}
Since $j_i \ge 0$ for $1 \le i \le s$, it follows that $\boldsymbol{F}$ is empty if $m-t-3s < 0$.

We have
\begin{eqnarray*}
\lefteqn{ e(H_{\alpha,\bsgamma}, Q_{P,\Lambda}) } \\ & \le &  \sum_{(\bsj,\bsl) \in
\boldsymbol{D} \setminus \boldsymbol{F}} 2^{-\alpha |\bsj|_1}
  + \sum_{\emptyset \neq u
\subseteq S} 2^{|u|/2} \sum_{(\bsj,\bsl) \in \boldsymbol{F}}  2^{-\alpha \sum_{i  = 1}^s j_{i}}
2^{-\alpha (m_{\bsd} - t)} \binom{m_{\bsd} - t + |u|}{|u|-1}.
\end{eqnarray*}
For the first sum we have
\begin{eqnarray*}
\sum_{(\bsj,\bsl) \in \boldsymbol{D}\setminus \boldsymbol{F}} 2^{-\alpha |\bsj|_1} & \le & 2^s \sum_{l = m-t-3s  + 1}^\infty 2^{-\alpha l} \sum_{\bsj \in \mathbb{Z}^s, |\bsj|_1 = l} 1 \\ & \le & 2^s \sum_{l=m-t-3s+1}^\infty 2^{-\alpha l} \binom{l + s-1}{s-1} \\ & \le & 2^{-\alpha (m-t)} 2^{s(2+3\alpha) -  \alpha} \binom{m-t-2s}{s-1},
\end{eqnarray*}
where we used \cite[Lemma~6]{CDP}.

Now we consider the second sum. First note that
\begin{equation*}
m_{\bsd} = m- \sum_{i=1}^s (j_{i,d_i} + 3) \le m - 3s.
\end{equation*}
Hence we have
\begin{eqnarray*}
\lefteqn{ \sum_{(\bsj,\bsl) \in \boldsymbol{F}} 2^{-\alpha \sum_{i\in u} j_i} 2^{-\alpha (m_{\bsd}-t)} \binom{m_{\bsd} - t + |u|}{|u| - 1} } \\ & \le  & 2^{-\alpha (m-t)} \sum_{(\bsj,\bsl) \in \boldsymbol{F}} 2^{-\alpha \sum_{i = 1}^s j_i} 2^{\alpha \sum_{i=1}^s (j_i+3 )} \binom{m- t - 3s + |u|}{|u|-1} \\ & \le & 2^{-\alpha (m-t)} 2^{3 \alpha s } \binom{m-t -2s}{s-1} |\boldsymbol{F}|.
\end{eqnarray*}
We have
\begin{equation*}
|\boldsymbol{F}| \le 2^{2s} \sum_{l=0}^{m-t-3s} \binom{l+s-1}{s-1} = 2^{2s} \binom{m-t-2s}{s}.
\end{equation*}
Thus we obtain
\begin{eqnarray*}
\lefteqn{ e(H_{\alpha,\bsgamma}, Q_{P,\Lambda})  } \\ & \le & 2^{-\alpha (m-t)} 2^{s(3\alpha + 2) - \alpha} \binom{m-t-2s}{s-1} + 2^{-\alpha (m-t)} 2^{s(3 \alpha + 2) + 1} \binom{m-t-2s}{s}^2 \sum_{\emptyset \neq u \subseteq S} 2^{|u|/2} \\ & \le & 2^{-\alpha (m-t)} \binom{m-t-2s}{s}^2 2^{s(3\alpha+2)} \left(2^{- \alpha} + 2 (1 + 2^{1/2})^s \right)
\end{eqnarray*}
and hence the result follows.
\end{proof}


\subsection{Exponential decay of the weight parameters}\label{sec_examples}

Now we consider numerical integration of functions $f:\mathbb{R}^s \to \mathbb{R}$ where the functions and the modulus of continuity decay exponentially fast.

Let the local smoothness parameters $\bsalpha$ be given by
\begin{equation*}
\alpha_u(\bsx) = \alpha,
\end{equation*}
for all $\emptyset \neq u \subseteq S$, where $1/2 < \alpha \le 1$.

For $J = \prod_{i=1}^s [a_i, b_i)$ let $\gamma_i(J) = 2^{-\min(|a_i|,|b_i|)}$ and
\begin{equation*}
\gamma_u(J) = \prod_{i =1}^s \gamma_i(J)
\end{equation*}
for $u \subseteq S$.

Let $b = 2$ and $m_l = m-1 - \lceil l/2 \rceil$ for $1 \le l \le 2(m-2)$ and $m_{2m-3} = m_{2m-2} = 1$. Then
\begin{equation*}
\sum_{l=1}^{2m-2} 2^{m_l} = 2[1 + 1 + 2 + 2^2 + \cdots + 2^{m-2}] = 2^m.
\end{equation*}

For $1 \le i \le s$ and $d_i \in \mathbb{Z}$ let $j_{i,d_i} = 0$,
\begin{equation*}
l_{i,d_i} = \left\{\begin{array}{ll} (d_i-1)/2 & \mbox{for } d_i \ge 1, d_i \mbox{ odd}, \\ - d_i/2 & \mbox{if } d_i \ge 2, d_i \mbox{ even}, \end{array} \right.
\end{equation*}
and
\begin{equation*}
J_{i,d_i} = [l_{i,d_i},  (l_{i,d_i} + 1)).
\end{equation*}
The intervals $J_{i,1}, J_{i,2}, J_{i,3}, \ldots$ form a partition of $\mathbb{R}$. Let $$E_i = \{(j_{i,d_i}, l_{i,d_i}): 1 \le d_i \le 2m-2, 1 \le i \le s\}$$ and $\boldsymbol{E} = \prod_{i=1}^s E_i$. Let $$D_i = \{(j_{i,d_i}, l_{i,d_i}): d_i \ge 1, 1 \le i \le s\}$$ and $\boldsymbol{D} = \prod_{i=1}^s D_i$.

Let $P$ be the point set constructed according to Section~\ref{sec_construction} based on a digital $(t,m,s)$-net over $\mathbb{Z}_2$.

\begin{corollary}\label{cor_expdecay}
Let $P$ be constructed using the scheme of Section~\ref{sec_construction} based on a digital $(t,m,s)$-net over $\mathbb{Z}_2$ where $m > t + 2s$ and $s \ge 2$.
With the choice of parameters as above we have
\begin{equation*}
e(H_{\bsalpha,\bsgamma}, Q_{P,\Lambda}) \le 2^{3s-1} 2^{-(m-t)} \binom{m-t}{s-1} + 2^{s(2\alpha + 1)} 2^{-\alpha (m-t)} \binom{m-t-s}{s}^2.
\end{equation*}
\end{corollary}

\begin{proof}
We have
\begin{equation*}
\gamma_{i}(J_{i,d_i}) = 2^{- \min(|l_{i,d_i}|, |l_{i,d_i}+1|)}.
\end{equation*}
Therefore, for $u \subseteq S$ and $J_{\bsd} = \prod_{i=1}^s J_{i,d_i}$ we have
\begin{equation*}
\gamma_{u}(J_{\bsd}) = 2^{-\sum_{i=1}^s \min(|l_{i,d_i}|, |l_{i,d_i}+1|}.
\end{equation*}

Further we have
\begin{equation*}
C_{\alpha,|u|} \le 2^{|u|/2}.
\end{equation*}

We have $m_{i,d_i} = m-1 - \lceil d_i/2\rceil$ and hence $m-m_{i,d_i} = 1 + \lceil d_i/2 \rceil$ for $1 \le d_i \le 2(m-2)$ and $m_{i,2m-3} = m_{i,2m-2} = 1$. Thus $m-m_{i,d_i} = \min(1+\lceil d_i/2 \rceil, m-1)$ for $1 \le d_i \le 2m-2$. From $m_{\bsd} \ge t$ it follows that
\begin{equation*}
t \le m_{\bsd} = m-\sum_{i=1}^s (m-m_{i,d_i}) = m-\sum_{i=1}^s \min(1+\lceil d_i/2\rceil, m-1),
\end{equation*}
which implies that
\begin{equation}\label{ineq_ex_exp}
m-t \ge \sum_{i=1}^s \min\left(1 + \lceil d_i/2 \rceil,m-1 \right) = \sum_{i=1}^s \min(1+ \max(|l_{i,d_i}|, |l_{i,d_i}+1|), m-1),
\end{equation}
since $\lceil d_i/2 \rceil = l_{i,d_i}+1$ for $d_i$ odd and $\lceil d_i/2 \rceil = - l_{i,d_i}$ for $d_i$ even. Since $m-t \le m$ and $1+\max(|l_{i,d_i}|, |l_{i,d_i}+1|) \ge 1$, it follows that $1+\max(|l_{i,d_i}|, |l_{i,d_i}+1|) \le m-1$ in \eqref{ineq_ex_exp}. Hence
\begin{equation*}
\boldsymbol{F} = \left\{(\bsj,\bsl) \in \boldsymbol{E}: \sum_{i=1}^s \max(|l_{i,d_i}|, |l_{i,d_i}+1|) \le m- t-s \right\}.
\end{equation*}

Therefore we have
\begin{eqnarray*}
\lefteqn{ e(H_{\alpha,\bsgamma}, Q_{P,\Lambda})  \le  \sum_{(\bsj,\bsl) \in
\boldsymbol{D} \setminus \boldsymbol{F}} 2^{- \sum_{i=1}^s \min(|l_{i,d_i}|, |l_{i,d_i}+1|)} } \\ &&  + \sum_{\emptyset \neq u
\subseteq S} 2^{|u|/2} \sum_{(\bsj,\bsl) \in \boldsymbol{F}} 2^{- \sum_{i=1}^s \min(|l_{i,d_i}|, |l_{i,d_i}+1|)}  2^{-\alpha (m_{\bsd} - t)} \binom{m_{\bsd} - t + |u|}{|u|-1}.
\end{eqnarray*}
For the first sum we have
\begin{eqnarray*}
\sum_{(\bsj,\bsl) \in
\boldsymbol{D} \setminus \boldsymbol{F}}  2^{- \sum_{i=1}^s \min(|l_{i,d_i}|, |l_{i,d_i}+1|)}  & \le & 2^{s} \sum_{l=m-t-s+1}^\infty 2^{-l} \binom{l+s-1}{s-1} \\ & \le & 2^{3s -1} 2^{-(m-t)} \binom{m-t}{s-1},
\end{eqnarray*}
where we used \cite[Lemma~6]{CDP}.

Now we consider the second sum. First note that
\begin{equation*}
m_{\bsd} = m- \sum_{i=1}^s \min(1 + \lceil d_i/2 \rceil,m-1) \le m - 2s.
\end{equation*}
Hence we have
\begin{eqnarray*}
\lefteqn{\sum_{(\bsj,\bsl) \in \boldsymbol{F}} 2^{-\sum_{i=1}^s \min(|l_{i,d_i}|, |l_{i,d_i}+1|)}  2^{-\alpha (m_{\bsd} - t)} \binom{m_{\bsd} - t + |u|}{|u|-1}} \\ & \le &  2^{-\alpha (m-t)} \sum_{(\bsj,\bsl) \in \boldsymbol{F}} 2^{- \sum_{i=1}^s \min(|l_{i,d_i}|, |l_{i,d_i}+1|)} 2^{\alpha \sum_{i=1}^s (1 + \max(|l_{i,d_i}|, |l_{i,d_i}+1|))} \binom{m-t-s}{s-1} \\ & \le & 2^{-\alpha(m-t)} 2^{2\alpha s} \binom{m-t-s}{s-1} |\boldsymbol{F}|,
\end{eqnarray*}
as $-\min(|l_{i,d_i}|, |l_{i,d_i}+1|) + \alpha \max(|l_{i,d_i}|, |l_{i,d_i}+1|) \le \alpha$.

We have
\begin{equation*}
|\boldsymbol{F}| \le 2^s \sum_{l=0}^{m-t-2s} \binom{l+s-1}{s-1} \le 2^s \binom{m-t-s}{s}.
\end{equation*}
Thus we obtain
\begin{eqnarray*}
\lefteqn{ e(H_{\alpha,\bsgamma}, Q_{P,\Lambda})  } \\ & \le & 2^{3s-1} 2^{-(m-t)} \binom{m-t}{s-1} + 2^{s(2\alpha + 1)} 2^{-\alpha (m-t)} \binom{m-t-s}{s}^2
\end{eqnarray*}
and hence the result follows.
\end{proof}


\section{Numerical result}\label{sec_numresult}

In this section we present some numerical results. We tested the
algorithm for the approximation of the integral
\begin{equation*}
I = \int_{-\infty}^\infty \int_{-\infty}^\infty
\int_{-\infty}^\infty \mathrm{e}^{2\sqrt{\pi} (x+y+z)}
\mathrm{e}^{-\pi (x^2+y^2+z^2)} \,\mathrm{d} x \,\mathrm{d} y
\,\mathrm{d} z = \left(e \int_{-\infty}^\infty
\mathrm{e}^{-(1-\sqrt{\pi} x)^2} \,\mathrm{d} x\right)^3 =
\mathrm{e}^3.
\end{equation*}

As underlying digital net we use the first $2^m$ points of a Sobol
sequence. As in the examples in Subsection~\ref{subsec_example2} and
\ref{sec_examples}, we set $m_l= m-1-\lceil l/2 \rceil$ for $1 \le l
\le 2(m-2)$ and $m_{2m-3} = m_{2m-2}=1$. Then
\begin{equation*}
\sum_{l=1}^{2m-2} 2^{m_l} = 2[1 + 1 + 2 + 2^2 + \cdots + 2^{m-2}] =
2^m.
\end{equation*}

The partitioning of the real line can be done in different ways. For
instance, let
\begin{equation*}
a_l = \mathrm{erfinv}(1-2^{-l})*X \quad \mbox{for } 0 \le l < m,
\end{equation*}
where $\mathrm{erfinv}$ denotes the inverse error function. The number $X$ determines the size of the region where the integrand will be estimated. In the numerical experiments below we consider two values for $X$. Then for
$1 \le l \le m-1$ set
\begin{equation*}
J_{2l-1} = [a_{l-1}, a_l)
\end{equation*}
and
\begin{equation*}
J_{2l} = [-a_l, -a_{l-1}).
\end{equation*}
In each interval $J_l$ we define $2^{m_l}$ equally spaced and
left-centered points, which we use as the one dimensional
projections.

If a point $\bsx_n \in \prod_{i=1}^3 J_{l_i}$, then the
corresponding weight $\lambda_n$ is given by
\begin{equation*}
\lambda_n = \prod_{i=1}^s (a_{l_i}-a_{l_i-1}) N^{-1}_{l_1,l_2,l_3},
\end{equation*}
where $N_{l_1,l_2,l_3}$ is the number of points in the interval
$\prod_{i=1}^s [a_{l_i}, a_{l_i-1})$. (Note that if $N_{l_1,l_2,l_3}
< 2^t$, then the corresponding weight can be arbitrarily chosen in
the interval $[0,\prod_{i=1}^s (a_{l_i}-a_{l_i-1})
N^{-1}_{l_1,l_2,l_3}]$.)

A Matlab implementation of this algorithm can be found at
\begin{quote} \textrm{http://quasirandomideas.wordpress.com/2010/11/11/qmc-rules-over-rs-matlab-code-and-numerical-example/}
\end{quote}

We compare the result using the proposed algorithm with using the inverse normal cumulative distribution function, that is, we write the integral as
\begin{equation*}
I = \int_0^1 \int_0^1 \int_0^1 \mathrm{e}^{2 \sqrt{\pi} [\mathrm{erfinv}(2x - 1) + \mathrm{erfinv}(2y-1) + \mathrm{erfinv}(2z-1)]} \,\mathrm{d} x \,\mathrm{d} y \,\mathrm{d} z.
\end{equation*}
We approximate $I$ by
\begin{equation*}
Q = \frac{1}{N} \sum_{n=0}^{N-1} \mathrm{e}^{2 \sqrt{\pi} [\mathrm{erfinv}(2x_n - 1) + \mathrm{erfinv}(2y_n-1) + \mathrm{erfinv}(2z_n-1)]},
\end{equation*}
where $\bsx_n = (x_n,y_n,z_n) \in [0,1)$, $0 \le n < N$ are the quadrature points.

We use the first $2^m$ points of a Sobol sequence as underlying digital net in both cases. 

The errors and computation times are listed in Table~\ref{table1}. There, e(Rs) denotes the error using the proposed method, e(invcom) denotes the error using the inverse normal cumulative distribution function, t(Rs) denotes the time required for the computation using the proposed method in seconds, and t(invcom) denotes the time required using the inverse normal cumulative distribution function in seconds.

The speedup in terms of the computation time for the proposed algorithm comes from the fact that the computation of the inverse normal cumulative distribution function is computationally intensive and the fact that this step is not needed in our method, but is needed in the comparison method. In problems where computing the inverse cumulative distribution function is not a significant factor, then one can expect the computation times to be more similar to each other.

The integration error is also significantly lower using the proposed method. There is however a dependence on the parameter $X$. This value determines the size of the sample space, which is an important ingredient in the computation. Using a different partitioning numbers $a_0,\ldots, a_{m-1}$ could yield even better results. In high dimensions, say $s > 5$, it is likely to be advisable to use a different partitioning for coordinate indices beyond, say, $5$. A complete investigation of good choices of partitionings for high dimensional problems is left for future work.

\begin{table}
\begin{center}
\begin{tabular}{l|lll|lll}
m & e(Rs) $X=6$ & e(Rs) $X=12$ & e(invcom) & t(Rs) $X=6$ & t(Rs) $X=12$ & t(invcom) \\ \hline
13 & 0.139001 & 1.970174 & 0.699538 & 0.031 & 0.032 & 0.289 \\ 
14 & 0.232291 & 5.566163 & 0.290106 & 0.035 & 0.036 & 0.575 \\ 
15 & 0.216679 & 0.828577 & 0.380826 & 0.070 & 0.071 & 1.257 \\ 
16 & 0.015490 & 0.233993 & 0.398023 & 0.140 & 0.140 & 2.414 \\ 
17 & 0.072803 & 0.408627 & 0.323633 & 0.282 & 0.283 & 4.732 \\ 
18 & 0.024119 & 0.114013 & 0.298877 & 0.587 & 0.587 & 9.752 \\ 
19 & 0.026249 & 0.064150 & 0.141170 & 1.308 & 1.315 & 19.460 \\ 
20 & 0.000056 & 0.115068 & 0.144152 & 2.878 & 2.946 & 37.755 \\ 
21 & 0.000002 & 0.003248 & 0.120121 & 6.229 & 5.970 & 75.6635 \\
22 & 0.000199 & 0.002157 & 0.096750 & 12.250 & 12.538 & 153.900 \\ \hline
\end{tabular}
\end{center}
\label{table1}
\end{table}

\end{document}